\documentclass[10pt]{article}

\usepackage[margin=1.2in]{geometry}
\usepackage{framed,multirow}

\usepackage{graphicx}
\usepackage{latexsym, amssymb, enumerate, amsmath}
\usepackage{cite}
\usepackage{defs}
\usepackage{fonts}
\usepackage{epsfig}
\usepackage{epstopdf}
\usepackage{booktabs}
\usepackage{lineno}
\usepackage{xfrac}
\usepackage{hyperref}
\usepackage{amsthm}
\newtheorem{theorem}{Theorem}[section]
\newtheorem{lemma}[theorem]{Lemma}

\usepackage{subfigure}
\usepackage{enumitem}
\usepackage[ruled,vlined,linesnumbered]{algorithm2e}
\usepackage{eufrak,mathrsfs, euscript, mathdesign}

\usepackage{float}

\usepackage{xargs}                      
\usepackage{upgreek}

\usepackage{comment} 

\providecommand{\norm}[1]{\lVert #1 \rVert}

\usepackage{url}

\usepackage{todonotes}

\title{Data-driven, structure-preserving approximations to entropy-based moment closures for kinetic equations\thanks
	{This manuscript has been authored, in part, by UT-Battelle, LLC, under Contract No. DE-AC0500OR22725 with the U.S. Department of Energy. 
		The United States Government retains and the publisher, by accepting the article for publication, acknowledges that the United States Government retains a non-exclusive, paid-up, irrevocable, world-wide license to publish or reproduce the published form of this manuscript, or allow others to do so, for the United States Government purposes. The Department of Energy will provide public access to these results of federally sponsored research in accordance with the DOE Public Access Plan (\texttt{http://energy.gov/downloads/doe-public-access-plan}).}}

\date{\today}
\author{
	William A. Porteous%
\thanks{Department of Mathematics, 
	University of Texas at Austin, 
	Austin, TX 78712 USA, (\texttt{afpwilliam@gmail.com}).
}
\and
		M. Paul Laiu%
\thanks{Multiscale Methods and Dynamics Group,
				Computer Science and Mathematics Division,
				Oak Ridge National Laboratory,
				Oak Ridge, TN 37831 USA, (\texttt{laiump@ornl.gov}).
}       
\and
Cory D. Hauck%
\thanks{Multiscale Methods and Dynamics Group,
	Computer Science and Mathematics Division,
	Oak Ridge National Laboratory,
	Oak Ridge, TN 37831 USA, (\texttt{hauckc@ornl.gov}).
}
}

\begin{document}

\maketitle

\begin{abstract}
We present a data-driven approach to construct entropy-based closures for the moment system from kinetic equations.
The proposed closure learns the entropy function by fitting the map between the moments and the entropy of the moment system, and thus does not depend on the space-time discretization of the moment system and specific problem configurations such as initial and boundary conditions.
With convex and $C^2$ approximations, this data-driven closure inherits several structural properties from entropy-based closures, such as entropy dissipation, hyperbolicity, and H-Theorem. 
We construct convex approximations to the Maxwell-Boltzmann entropy using convex splines and neural networks, test them on the plane source benchmark problem for linear transport in slab geometry, and compare the results to the standard, optimization-based M$_N$ closures.
Numerical results indicate that these data-driven closures provide accurate solutions in much less computation time than the M$_N$ closures.
\end{abstract}



\section{Introduction}
\label{sec:intro}

Kinetic equations describe the movement of particles and their interactions with each other and with the surrounding environment.  They are used to model dilute particle systems in non-equilibrium regimes, with applications that includes neutral gases \cite{cercignani1988boltzmann}, plasmas \cite{hazeltine2018framework}, radiation transport \cite{lewisMiller_1993, pomraning_1973}, and charge transport in materials \cite{markowich2012semiconductor}.  The solution of a kinetic equation is a  density defined over position-velocity phase space.  Weighted integrals of this density with respect to velocity recover important physical quantities such as mass, momentum, and energy densities.

In many applications, kinetic equations are too expensive to simulate directly, and reduced models are needed.  Moment methods are a class of reduced models which approximate the evolution of a finite number of velocity averages, or moments, of the kinetic distribution.   The form of such models depends on the closure which  the information from the kinetic density that is lost in the moment approach.    The structure and behavior of such models depends heavily on the choice of closure.

Entropy-based closures  \cite{levermore_1996} approximate the unknown kinetic density by the solution of a convex minimization (using the convention of convex entropy, rather than concave)  problem that is constrained by the known moments which are input from the model.    These models inherit many of the structural features of kinetic equations, including hyperbolicity, entropy dissipation, and an H-Theorem which relates the uniquely characterizes equilibrium states as points where the entropy dissipation vanishes.

Despite their elegant structure, the practical implementation of entropy-based moment closures presents three major challenges.  First, in some degenerate cases, the entropy minimization  problem (see \eqref{eq:primal} below) may not have a solution \cite{hauck2008convex, junk2000maximum, caflisch1986equilibrium, schneider2004entropic}.  Roughly speaking a finite infimum exists, but it is never attained.   Second, even in cases that the defining optimization problem is well-posed, a complicated and expensive numerical procedure is required to solve it via the convex dual (see  \eqref{eq:dual} below) \cite{hauck2011high,alldredge2012high,alldredge2014adaptive,garrett2015optimization,abramov2007improved, abramov2010multidimensional, abramov2009multidimensional}.  Third, a necessary condition for the entropy optimization problem to have a solution is that the moments in the constraint set be realizable; that is, they are moments of a non-negative kinetic distribution.  However, numerical methods which preserve convex invariant domains (such as the realizable set) can be challenging to construct.   Limiters based on the strategy proposed in  \cite{zhangShu_2010}  have been used in \cite{olbrant2012realizability, chu2019realizability,alldredge2015realizability}.  Another viable limiting strategy can be found in \cite{guermond2019invariant}.  

An alternative to limiting in a numerical method is to regularize the optimization problem.  Such a  strategy was introduced in \cite{alldredge2019regularized}.  The main idea is to replace the moment constraints in the optimization problem by a penalization term.  The resulting moment system maintains hyperbolicity and entropy dissipation properties of the original entropy-based closure, but extends the realizable set to the entire space. Recently, a rigorous convergence analysis has been initiated \cite{alldredge2021convergence} to understand the errors made in the regularization.  

The relaxed minimization problem that defines the regularized closure is still numerically expensive to solve.  To address this problem, we take advantage of the entropy structure to construct a data-driven approximation to the entropy-based closure. The development of data-driven closures has been an active area of research lately.  Constructive models include the classical work in \cite{kershaw1976flux}, as well as various approximations of entropy-based closures \cite{sarr2020second,pichard2017approximation,mcdonald2013affordable}.  Learned models (i.e., those which rely on optimization of a loss function) can be found in \cite{bois2020neural,huang2021machine,han2019uniformly,huang2021machine2,sadr2020gaussian}.   Of these learned models the approach presented here is most closely related to the work in \cite{huang2021machine2,sadr2020gaussian}.  The approach in \cite{sadr2020gaussian}, which is applied to  the Boltzmann equation of gas-dynamics, accepts the entropy-based closure as sufficient and then seeks to approximate the map between the moments and the solution of the dual problem.  Thus the method does not require training from a full kinetic simulation, but instead relies only on the optimization problem.   The consequences of this difference are not entirely clear:  on one hand approximating the entropy-based closure is likely more robust, reproducible, flexible, and computationally efficient.   However, training with simulation data may provide important prior information that leads to higher fidelity simulations, if that information can be transferred to similar problem setups.  The method in \cite{huang2021machine2} uses kinetic simulation data to train the model, and is therefore not an approximation of the entropy-based closure.  However, it does enforce hyperbolicity, at least for linear kinetic models in slab geometries, like those studied in the numerical results section of this paper.   

In the current paper, we propose a model that learns the map between the moments and the entropy of the moment system, rather than the moments and the dual variables.  Even though the dual variables are the derivative of the entropy with respect to the moments, approximating them directly does not enforce the convex structure required to ensure entropy dissipation and hyperbolicity in the moment system.  However, arguments established in  \cite{alldredge2019regularized} show that a convex approximation to the entropy will preserve such properties.  With this fact in mind, we simulate the entropy-based systems M$_1$ and M$_2$ models which are used to approximate radiative transfer.  In our numerical results, we consider a simple linear kinetic in slab geometry, like the one use in \cite{huang2021machine2}.  However, the approach is generalizable to arbitrary moment order and dimension and to nonlinear problems.  The only requirement is that the entropy-based closure is well-posed and the approximation of the entropy is convex.

We construct convex approximations of the entropy using convex splines and neural networks.  In both cases, we rely on a normalization of the moment space in order to reduce the dimensions of the problem and make the domain bounded.  We show that for the Maxwell-Boltzmann entropy, this normalization preserves convexity and $C^2$ regularity of the entropy. The spline approximation is fast and provably convex everywhere in the domain.  However,  they are used only in the M$_1$ case, which has two moment but reduces to a one-dimensional fitting  problem after the normalization.  For M$_2$, the domain for the fitting problem is two-dimensional.  Although convex splines exist is this setting, we are not aware of a constructive approach that yields a $C^2$ function.    For the neural network approximation, convexity can be enforced at training points and checked at test points.  For the problems considered here, we have found that the neural network fit is convex on a densely sampled set of test points.  However, this approach is not viable in higher dimensions and networks which guarantee convexity \cite{amos2017input} will be necessary.

The rest of paper is organized as follows.  In Section \ref{sec:prelim}, we briefly introduce kinetic equations, entropy-based moments methods, and approximate entropy closures.  In Section \ref{sec:data_closures}, we describe the data-driven closures, using the Maxwell-Boltzmann entropy and applied to a simple linear kinetic equation.  In Section \ref{sec:num_results}, we present implementations details and numerical results.  Summary and conclusions are given in Section~\ref{sec:conclusion}.
\section{Preliminaries }
\label{sec:prelim}

\subsection{Kinetic equation}
\label{subsec:kinetic}

For neutral particle systems, the governing kinetic equation takes the form
\begin{equation}\label{eq:kinetic_eqn}
\p_t f(t,x,v) + v\cdot \nabla_x f = \cC(f(t,x,\cdot))(v)\:,
\end{equation}
where the evolving \textit{kinetic density function} $f\colon [0,\infty)\times X \times V \to \bbR_+$ describes the particle density at time $t\in\bbR_+:=[0,\infty)$ at position $x\in X\subseteq\bbR^d$ traveling at velocity $v\in V \subseteq\bbR^d$. 
The operator $v\cdot\nabla_x(\cdot) := \sum_{i=1}^d v^{[i]} \p_{x^{[i]}} (\cdot)$ models the advection of particles, where $v^{[i]}$ and $x^{[i]}$ respectively denotes the velocity and position component in dimension $i$.
The collision operator $\cC$ is an integral operator in $v$ at each $(t,x)$, which introduces the interactions between particles as well as the interaction of particle and the background medium.
In order to be well-posed, \eqref{eq:kinetic_eqn} must be equipped with appropriate initial and boundary conditions.   Solutions to \eqref{eq:kinetic_eqn} are often restricted to values in a set $B\subseteq \bbR_+$, which reflects physical bounds on $f$ such as positivity.

The collision operator is endowed with certain structural features \cite{levermore_1996} (see also \cite{alldredge2019regularized}):
\begin{enumerate}[label=(\roman*)]
	\itembf{Collision invariants}:  There exist function $e \colon V\to\bbR$ such that $\vint{ e \cC(g)} =0$ for all $g \in \operatorname{Dom}(\cC)$, where $\vint{\cdot}:=\int_V\cdot\, dv$.  As a result
	\begin{equation}
	\p_t \vint{e f} + \nabla_x \cdot \vint{v e f} = 0\:,
	\end{equation}
	where $ \nabla_x \cdot \vint{g}:=\sum_{i=1}^d \p_{x^{[i]}} \vint{g^{[i]}}$.  We denote the set of collision invariants associated to $\cC$ as $\bbE$.
	\itembf{Entropy dissipation}: For an open set $D\subseteq \bbR$, there exists a twice differentiable, strictly convex function $\eta\colon D\to\bbR$, referred to as the kinetic entropy density, such that 
	\begin{equation}\label{eq:entropy_disspation_condition}
	\vint{\eta^\prime(g) \cC(g)} \leq 0 \quad \text{for all $g$ in Dom($\cC$) such that Range($g$)$\subseteq D$}\:.
	\end{equation}
	Integrating \eqref{eq:kinetic_eqn} in $v$ against $\eta^\prime(f)$ then leads to the entropy dissipation law
	\begin{equation}
	\p_t \vint{\eta(f)} + \nabla_x \cdot \vint{v\eta(f)} = \vint{\eta^\prime(f) \cC(f)} \leq 0\:.
	\end{equation}
	\itembf{H-theorem}: For any $g \in \operatorname{Dom}(\cC)$ such that Range($g$)$\subseteq D$, the following statements are equivalent:
	\begin{equation}
	\label{eq:h_thm}
	\text{(a)}\,\,\vint{\eta^\prime(g)\cC(g)} = 0\:;\quad
	\text{(b)}\,\,\cC(g) = 0\:;\quad
	\text{(c)}\,\,\eta^\prime(g)\in\bbE\:.
	\end{equation}
\end{enumerate}

\subsection{Entropy-based moment methods}
\label{subsec:closures}

Moment methods approximate the evolution of a finite set of integral averages, i.e., moments, of $f$ given by 
\begin{equation}
\bu_f(t,x) = \vint{\bfm(\cdot) f(t,x,\cdot)},
\end{equation}
where the vector-valued function $\bfm(v) = [m_0(v),\dots,m_n(v)]^T$ collects a set of test functions in $v$.  In most cases, these test functions form the basis of the space $\bbP_N(V)$ consisting of polynomials on $V$ up to a prescribed degree $N$.  The moments $\bu_f$ satisfy the equation
\begin{equation}\label{eq:moment_eqn_exact}
\p_t \bu_f + \nabla_x \cdot \vint{v \bfm f} = \vint{\bfm\cC(f)},
\end{equation}
which is not closed.  A kinetic closure is one in which an approximation for $f$ is used to close the system.   Let $\bbF_{\bfm} = \{g \in \text{Dom}(\cC) \colon \text{Range}(g)\subseteq D ~\text{and}~ |\vint{\bfm g}| < \infty \}$ and let $\cR_\bfm = \{\bw: \bw = \vint{\bfm g} , g \in \bbF_{\bfm} \}$.  A kinetic closure is prescribed via an ansatz $F \colon \cR_\bfm \to \bbF_\bfm$ where $F_\bw$ is understood approximation of a kinetic distribution $g$ with moments $\bw \in \cR_\bfm$., i.e., $F_{\vint{\bfm g(x,t,\cdot})}(v) \approx g(x,v,t)$. The resulting (closed) moment system for $\bu(x,t) \approx \bu_f(x,t)$ take the form
\begin{equation}\label{eq:moment_eqn_F}
\p_t \bu+ \nabla_x \cdot \vint{v \bfm F_{\bu}} = \vint{\bfm\cC(F_{\bu})}.
\end{equation}
An entropy-based closure \cite{levermore_1996,dubroca1999etude} is one which specifies $F_\bw$ as the solution to the optimization problem
\begin{equation}\label{eq:primal}
\minimize_{g\in \bbF_\bfm} \cH(g):=\vint{\eta(g)} \qquad
\text{subject to}\quad \vint{\bfm g} = \bw\:.
\end{equation}
When a solution for \eqref{eq:primal} exists, it is unique and takes the form \cite{hauck2008convex, junk2000maximum} $G_{\hat\bsalpha(\bw)}$ where
\begin{equation}
\label{eq:G}
G_{\bsalpha}:= [\eta^*]^{\prime}(\bsalpha\cdot \bfm), \quad \bsalpha \in \bbR^{n+1} ;
\end{equation}
$\eta^*$ is the Legendre dual of $\eta$; and $\hat{\bsalpha}\colon\bbR^n\to\bbR^n$ is the map from $\bw$ to the solution of the dual problem to \eqref{eq:primal}, i.e., 
\begin{equation}\label{eq:dual}
\hat{\bsalpha}(\bw)=\argmax_{\bsalpha\in\bbR^n}\{\bsalpha\cdot\bw-\vint{\eta^*(\bsalpha\cdot\bfm)}\}\:.
\end{equation}
It follows from the equality constraints in \eqref{eq:primal} (or equivalently from the first-order optimality condition for \eqref{eq:dual}) that 
$
\vint{\bfm G_{\hat{\bsalpha}{(\bw)}}}= \bw\:,
$
i.e., the ansatz preserves moment vectors.
Thus the inverse of $\hat{\bsalpha}$, denoted as $\hat{\bw}\colon\bbR^n\to\bbR^n$, is given by
\begin{equation}\label{eq:moment_consistency}
\hat{\bw}(\bsalpha):=\vint{\bfm G_{\bsalpha}}\:.
\end{equation}
With this property, \eqref{eq:moment_eqn_F} becomes the moment equation
\begin{equation}\label{eq:moment_eqn_u}
\p_t \bu + \nabla_x \cdot\vint{v \bfm G_{\hat\bsalpha(\bu)}} = \vint{\bfm\cC(G_{\hat\bsalpha(\bu)})}\:.
\end{equation}
The system \eqref{eq:moment_eqn_u} is a symmetric hyperbolic balance law when expressed in terms of the dual variable $\bsbeta(x,t) = \hat \bsalpha(\bu(x,t))$ and that solutions to \eqref{eq:moment_eqn_u} dissipate the strictly convex entropy 
\begin{equation}\label{eq:entropy}
h(\bu):= \cH(G_{\hat\bsalpha(\bu)}),
\end{equation}
These results are easy to show once the known relationship $h'(\bw) = \hat\bsalpha(\bw)$ is established; see \cite{levermore_1996} for details.

\subsection{Approximate entropy-based closures}
\label{subsec:regularized_closures}

Let $h_{\textup{a}}$ be a strictly convex $C^2$ approximation of $h$ and let $h^*_{\textup{a}}$ be the Legendre dual of $h$, so that $[h^*_{\textup{a}}]' = [h'_{\textup{a}}]^{-1} $.  
For any $\bw \in \cR_\bfm$, define $\hat{\bsalpha}_{\textup{a}} (\bw)= h_{\textup{a}}'(\bw)$.  The approximate entropy-based closure for $\bu_{\textup{a}}$ is 
\begin{equation}\label{eq:reg_moment_eqn_u}
\p_t \bu_{\textup{a}} + \nabla_x \cdot\vint{v \bfm G_{\hat\bsalpha_{\textup{a}}(\bu_{\textup{a}} )}} = \vint{\bfm\cC(G_{\hat\bsalpha(\bu_{\textup{a}} )})}\:.
\end{equation}
The following result comes directly from \cite{alldredge2019regularized}, where $h_{\textup{a}}$ is defined to be $h_\gamma$, which is parameterized by a regularization parameter $\gamma > 0$ such that $\lim_{\gamma \to 0^+} h_\gamma = h$. (See \cite{alldredge2019regularized} for the explicit definition of $h_\gamma$.) Since it forms the basis of the data-driven closure strategy, we present it again here, along with a complete proof.

\begin{theorem}[{\cite{alldredge2019regularized}}]
	The moment system \eqref{eq:reg_moment_eqn_u} is symmetric hyperbolic in the variable $\bsbeta(t,x) := \hat \bsalpha_{\textup{a}}(\bu_{\textup{a}}(t,x) )$.  The function $h_{\textup{a}}$ acts as an entropy of \eqref{eq:reg_moment_eqn_u};  that is, $h_{{\textup{a}}}(\bu_{{\textup{a}}})$ is dissipated by solutions of \eqref{eq:reg_moment_eqn_u}. Moreover, the following statements are equivalent:
	\begin{equation}\label{eq:moment_H-thm}
	{\rm(a)}\,\, 
	\bsbeta \cdot \vint{\bfm\cC( G_{\bsbeta } ) } = 0
	\:;\quad
	{\rm(b)}\,\, 
	\vint{\bfm\cC(G_{\bsbeta }  )} = 0
	\:;\quad
	{\rm(c)}\,\, \bsbeta \cdot \bfm \in\bbE\:.
	\end{equation}
\end{theorem}

\begin{proof}
	By definition, $\bsbeta = h'_{\textup{a}} (\bu_{\textup{a}})$, which implies that $\bu_{\textup{a}}= [h^*_{\textup{a}}]'(\bsbeta)$.  Thus in terms of $\bsbeta$,  \eqref{eq:reg_moment_eqn_u} becomes, 
\begin{equation}\label{eq:reg_moment_eqn_beta}
\p_t [h^*_{\textup{a}}]'(\bsbeta) + \nabla_x \cdot\vint{v \bfm G_{\bsbeta} }= \vint{\bfm \cC(G_{\bsbeta} )}\:.
\end{equation}
The chain rules, along with the definition of $G$ in \eqref{eq:G}, then gives
\begin{equation}
\p_t [h^*_{\textup{a}}]'(\bsbeta)  = [h^*_{\textup{a}}]''(\bsbeta) \p_t \bsbeta
\quand
\nabla_x \cdot \Vint{v \bfm G_{\bsbeta} } = \Vint{v \bfm \, [\eta^*]'(\bsbeta\cdot \bfm)} \cdot \nabla_x \bsbeta
\end{equation}
The symmetric hyperbolic form then follows since $[h^*_{\textup{a}}]''(\bsbeta)$ is symmetric and positive definite and  $\Vint{v \bfm \bfm^T\, [\eta^*]'(\bsbeta\cdot \bfm)}$ is symmetric.   

To show entropy dissipation, we multiply \eqref{eq:reg_moment_eqn_u} by $\bsbeta$:
\begin{equation}
\label{eq:moment_entropy_diss}
\p_t h_{\textup{a}}(\bu_{\textup{a}}) + \bsbeta \cdot \nabla_x \cdot\vint{v \bfm G_{\bsbeta}} = \vint{(\bsbeta \cdot \bfm)\cC(G_{\bsbeta})}
\end{equation}
and observe that, upon integration by parts and the reverse chain rule,
\begin{equation}
\begin{aligned}
\bsbeta \cdot \nabla_x \cdot\vint{v \bfm G_{\bsbeta}} 
&= \nabla_x \cdot\vint{v (\bsbeta \cdot \bfm) G_{\bsbeta}} - \vint{v \bfm G_{\bsbeta}} \cdot \nabla_x \bsbeta \\
&= \nabla_x \cdot\vint{v (\bsbeta \cdot \bfm) G_{\bsbeta}} - \nabla_x  \cdot \vint{v \eta^*(\bsbeta \cdot \bfm) } 
\end{aligned}
\end{equation}
Hence there is a locally conservative entropy flux for smooth solutions.  Meanwhile, setting $[\eta^*]'(g):= \bsbeta \cdot \bfm$, we recognize
the source term in \eqref{eq:moment_entropy_diss} as
\begin{equation}
 \vint{(\bsbeta \cdot \bfm)\cC(G_{\bsbeta})} = [\eta^*]'(g)\cC(g) \leq 0
\end{equation}
This shows that the source terms dissipates the entropy.  Moreover, the conditions in \eqref{eq:moment_H-thm} are inherited directly from the kinetic structure formulated in \eqref{eq:h_thm}.
\end{proof}

\section{Data-driven approximations for the Maxwell-Boltzmann entropy}
\label{sec:data_closures}

In this paper, we focus on the construction of a data-driven approximate entropy to the Maxwell-Boltzmann entropy.
The Maxwell-Boltzmann entropy is commonly used to close the moment system due to its physical relevance. 
However, its application is often limited by the high computational cost incurred in the optimization procedure for calculating the ansatz.
The Maxwell-Boltzmann kinetic entropy density is given by 
\begin{equation}\label{eq:MB_entropy}
\eta_{\MB}(z):=z\log z - z\:,\quad\text{with}\quad
\eta_{\MB}^*(y) = [\eta_{\MB}^*]^\prime(y) = e^y\:.
\end{equation} 
The domain of $\eta_{\MB}$ is $D=\bbR_+$, thus the Maxwell-Boltzmann entropy leads to a nonnegative ansatz.
Specifically, plugging $\eta = \eta_{\MB}$ into \eqref{eq:primal}--\eqref{eq:dual} gives the ansatz $G_{\hat{\bsalpha}_{\MB}(\bw)}= e^{\hat{\bsalpha}_{\MB}(\bw)\cdot\bfm}$, where 
\begin{equation}\label{eq:MB_optimization}
\hat{\bsalpha}_{\MB}(\bw)= \argmax_{\bsalpha\in\bbR^n}\{\bsalpha\cdot\bw-\vint{e^{\bsalpha\cdot\bfm}}\}\:.
\end{equation}
The Maxwell-Boltzmann entropy $h_{\MB}$ is defined as $h_{\MB}:=\vint{\eta_{\MB}^*(G_{\hat{\bsalpha}_{\MB}(\bw)})}$, and the associated realizable set 
\begin{equation}\label{eq:MB_realizable_set}
\cR_{\bfm,\MB}:=\{\bw\in\bbR^n \colon \bw=\vint{\bfm g},\, g\geq0,\, \vint{|\bfm g|} <\infty  \}\:.
\end{equation}
It is known \cite{alldredge2012high,LHMOT-2016,GH12} that the optimization procedure for evaluating $\hat{\bsalpha}_{\MB}(\bw)$ at given moment vector $\bw$ is usually the most computationally intensive part in solving the moment system with Maxwell-Boltzmann moment closure.
We aim to construct a twice differentiable, convex approximate entropy $\hA\colon\bbR^n\to\bbR$ to the Maxwell-Boltzmann entropy $h_{\MB}$, and approximate the optimal multiplier $\hat{\bsalpha}_{\MB}$ by 
\begin{equation}
\halphaA(\bw):=\hA^\prime(\bw)\:.
\end{equation}
With a twice differentiable and convex $\hA$, the closed moment system inherits important features of the kinetic equation, including hyperbolicity, collisional invariants, entropy dissipation, and H-theorem , as shown in Section~\ref{subsec:regularized_closures}.
Further, since the proposed closure is based on an approximate entropy $\hA$, it can be constructed using data from the optimization problem \eqref{eq:MB_optimization}, which does not require full kinetic simulations. Thus, the resulting closure is independent of the space-time discretization of moment models, initial and boundary conditions, and collision operators. The closure only depends on the moment order $N$. 

We propose to construct $\hA$ based on a given dataset $\{(\bw^{(i)},h_{\MB}(\bw^{(i)}),\hat{\bsalpha}_{\MB}(\bw^{(i)}))\}_{i\in \cI}$, where $\{h_{\MB}(\bw^{(i)})\}_{i\in \cI}$ and $\{\hat{\bsalpha}_{\MB}(\bw^{(i)})\}_{i\in \cI}$ are respectively the function values and the gradients of the Maxwell-Boltzmann entropy, $h_{\MB}$, evaluated at the moments $\{\bw^{(i)}\}_{i\in \cI}$ sampled from the domain of $h_{\MB}$, i.e., the realizable set $\cR_{\bfm,\MB}$.
For the remainder of the paper, we focus on Maxwell-Boltzmann entropy and will drop the $\MB$ subscript for simplicity.

\subsection{Dimension reduction via normalization}
\label{subsec:normalization}

One key factor in the construction of data-driven closures is a good sampling strategy of the dataset $\{(\bw^{(i)},h(\bw^{(i)}),\hat{\bsalpha}(\bw^{(i)}))\}_{i\in \cI}$.
Due to the nonlinear nature of the entropy function, the quality of data-driven approximations depends on whether the sampled dataset provides a reasonable coverage of the realizable set $\cR_{\bfm}$. 
In the context of Maxwell-Boltzmann entropy, the realizable set $\cR_{\bfm}$ is unbounded, e.g., $\bw=[w_0,0,\dots,0]^T$ is in $\cR_{\bfm}$ for any $w_0>0$. 
The unboundedness of $\cR_{\bfm}$ makes sampling a representative dataset difficult.
In this section, we show that this issue can be addressed by (i) sampling data on a ``normalized" bounded realizable set $\tilde{\cR}_{\bfm}\subset\bbR^{n-1}$, (ii) constructing a convex approximation of the Maxwell-Boltzmann entropy on $\tilde{\cR}_{\bfm}$, and (iii) extending the constructed approximation to the full realizable set $\cR_{\bfm}\subset\bbR^n$.

From here on, we introduce the notation $w_0$ for the zero-th order moment and $\tilde{\bw}$ for the remainder moment vector, i.e., $\bw=[w_0, \tilde{\bw}^T]^T$. With this notation, we consider the normalized realizable set to be $\tilde{\cR}_{\bfm}:=\{\tilde{\bsomega}\in\bbR^{n-1}\colon \bsomega=[1,\tilde{\bsomega}^T]^T\in\cR_{\bfm} \}$, where each vector in $\tilde{\cR}_{\bfm}$ is a realizable moment vector with the zero-th moment normalized to one. 
The set $\tilde{\cR}_{\bfm}$ is known to be bounded for reasonable choices of $\bfm$, e.g. polynomial basis and spherical harmonic basis, which makes it much easier to sample a representative dataset on $\tilde{\cR}_{\bfm}$.
To this point, we aim to approximate $\tilde{h}$, the restriction of Maxwell-Boltzmann entropy $h$ on $\tilde{\cR}_{\bfm}$, by a data-driven convex approximation $\thA\in C^2(\tilde{\cR}_{\bfm})$ with data collected on $\tilde{\cR}_{\bfm}$.
The function $\thA$ is then extended from $\tilde{\cR}_{\bfm}$ to the full realizable set $\cR_{\bfm}$ to approximate $h$.

We first show in the following lemma that, for each $\bw\in\cR_{\bfm}$, the Maxwell-Boltzmann entropy $h(\bw)$ can be calculated from $w_0$ and $\tilde{h}(\tilde{\bw}/w_0)$, where $\tilde{h}$ is the restriction of $h$ on $\tilde{\cR}_{\bfm}$. 
\begin{lemma}\label{lem:extension}
	Let $h$ be the Maxwell-Boltzmann entropy, $\tilde{h}$ be the restriction of $h$ on $\tilde{\cR}_{\bfm}$, i.e., 
	\begin{equation}\label{eq:restriction}
	\tilde{h}(\tilde{\bsomega}):= h([1,\tilde{\bsomega}^T]^T) \quad\text{for all } \tilde{\bsomega}\in\tilde{\cR}_{\bfm},
	\end{equation}
	and $\tilde{\bsalpha}$ be the gradient of $\tilde{h}$.
	Then, for any moment vector $\bw=[w_0,\tilde{\bw}^T]^T\in\cR_{\bfm}$, it follows that
	\begin{equation}\label{eq:extension_exact}
	h(\bw) = w_0  \,\tilde{h}(\tilde{\bw}/w_0) + w_0\log w_0\:,
	\end{equation}	
	and
	\begin{equation}
	\label{eq:alpha_scaling_exact}
	\hat{\bsalpha}(\bw):=h^\prime(\bw) = \left[\begin{array}{c}
	\tilde{h}(\tilde{\bw}/w_0) + \frac{1}{w_0} \tilde{\bw}^T\tilde{\bsalpha}(\tilde{\bw}/w_0) + \log w_0 + 1\\
	\tilde{\bsalpha}(\tilde{\bw}/w_0)
	\end{array}\right]\:.
	\end{equation}
\end{lemma}
\begin{proof}
It follows from the definition of $h$ in \eqref{eq:entropy} and strong duality of the primal-dual pair \eqref{eq:primal}-\eqref{eq:dual} that 
\begin{equation}\label{eq:L1proof1}
h(\bw) = \hat{\bsalpha}(\bw)\cdot\bw - \vint{e^{\hat{\bsalpha}(\bw)\cdot\bfm}}\quand
\tilde{h}(\tilde{\bw}/w_0) = {h}({\bw}/w_0) = \hat{\bsalpha}(\bw/w_0)\cdot\bw/w_0 - \vint{e^{\hat{\bsalpha}(\bw/w_0)\cdot\bfm}}\:,
\end{equation}
From \eqref{eq:moment_consistency}, we have
\begin{equation}
\bw/w_0 = \hat{\bw}(\hat{\bsalpha}(\bw/w_0)) = \vint{\bfm\, e^{\hat{\bsalpha}(\bw/w_0)\cdot\bfm}}\:,
\end{equation}
which implies
\begin{equation}
\hat{\bw}(\hat{\bsalpha}(\bw))  = \bw = w_0\vint{\bfm\, e^{\hat{\bsalpha}(\bw/w_0)\cdot\bfm}} 
= \vint{\bfm\, e^{\hat{\bsalpha}(\bw/w_0)\cdot\bfm + \log w_0}}\:.
\end{equation}
Thus $\hat{\bsalpha}(\bw) = \hat{\bsalpha}(\bw/w_0) + [\log w_0,0,\dots,0]^T$. Plugging this into \eqref{eq:L1proof1} leads to
\begin{equation}
h(\bw) = \hat{\bsalpha}(\bw/w_0)\cdot\bw + w_0\log w_0 - \vint{e^{\hat{\bsalpha}(\bw/w_0)\cdot\bfm + \log w_0}}
= w_0 \,\tilde{h}(\tilde{\bw}/w_0) + w_0\log w_0\:.
\end{equation}
Then \eqref{eq:alpha_scaling_exact} can be verified by taking the gradient of $h$.
\end{proof}

The result in Lemma~\ref{lem:extension} allows us to construct an approximation $\thA$ to $\tilde{h}$ on the bounded set $\tilde{\cR}_{\bfm}$, and then approximate $h$ on $\cR_{\bfm}$ by the extension $\hA\colon\cR_{\bfm}\to\bbR$ defined as
\begin{equation}\label{eq:extension}
\hA(\bw) = w_0 \, \thA(\tilde{\bw}/w_0)  + w_0\log w_0\:, \quad\forall\bw\in\cR_{\bfm}\:.
\end{equation}
It is then left for us to prove that this $\hA$ is twice differentiable and convex with a properly chosen $\thA$, which is shown in the following theorem.

\begin{theorem}\label{thm:convexity}
Suppose that $\thA\colon \tilde{\cR}_{\bfm}\to\bbR$ is a (strictly) convex and twice differentiable function, then the reconstructed function $\hA$ defined in \eqref{eq:extension} is also (strictly) convex and twice differentiable.
\end{theorem}
\begin{proof}
First, the twice differentiability of $\hA$ directly follows from \eqref{eq:extension} and the twice differentiability of $\tilde{\hA}$.
To show that $\hA$ is convex when $\thA$ is convex, we prove that $H(\bw)$, the Hessian of $\hA$, is positive semidefinite for all $\bw\in\cR_{\bfm}$.
Let $H$ take the form
\begin{equation}
{H}(\bw) = 
\left[\begin{array}{cc}
a(\bw) &\bb^T(\bw)\\
\bb(\bw) &M(\bw)
\end{array}\right]\:,
\quad\text{with}\quad
a \colon\cR_{\bfm}\to\bbR\:, \quad
\bb \colon\cR_{\bfm}\to\bbR^{n-1}\:, \quad
M = \colon\cR_{\bfm}\to\bbR^{(n-1)\times(n-1)}\:.
\end{equation}
Then direct calculation of the second derivatives leads to
\begin{equation}\label{eq:hessianTerms}
a(\bw)  = \frac{1}{w_0^{3}}\tilde{\bw}^{T} \tilde{H}(\tilde{\bsomega}) \tilde{\bw}+\frac{1}{w_0} \:, \quad
\bb(\bw)  = -\frac{1}{w_0^2}\tilde{H}(\tilde{\bsomega}) \tilde{\bw}\:, \quand
M(\bw)  = \frac{1}{w_0}\tilde{H}(\tilde{\bsomega})\:,
\end{equation}
where $\bw=:[w_0,\tilde{\bw}^T]^T$, $\tilde{\bsomega}:=\tilde{\bw}/w_0$, and $\tilde{H}$ is the Hessian of $\thA$.
The terms in \eqref{eq:hessianTerms} are well-defined since $w_0>0$ for any $\bw\in\cR_{\bfm}$. 
We now show that $H$ is positive semi-definite by verifying that $\bp^T H(\bw) \bp\geq0$ for all $\bw\in\cR_{\bfm}$, $\bp\in\bbR^n$. Let $\bp=:[p_0,\tilde{\bp}]^T$, then
\begin{equation}\label{eq:hessianPSD}
\begin{alignedat}{2}
\bp^T H(\bw) \bp &= p_0^2 \,a(\bw) + 2\,p_0\, \tilde{\bp}^T\bb(\bw) + \tilde{\bp}^T M(\bw) \tilde{\bp}\\
&= \frac{p_0^2}{w_0^{3}}\tilde{\bw}^{T} \tilde{H}(\tilde{\bsomega}) \tilde{\bw}+\frac{p_0^2}{w_0}  -  2\,\frac{p_0}{w_0^2}\, \tilde{\bp}^T\tilde{H}(\tilde{\bsomega}) \tilde{\bw} + \frac{1}{w_0}\tilde{\bp}^T \tilde{H}(\tilde{\bsomega}) \tilde{\bp}\\
&=\frac{p_0^2}{w_0}  + \frac{1}{w_0}
\left(\frac{p_0}{w_0}\tilde{\bw} - \tilde{\bp}\right)^{T} \tilde{H}(\tilde{\bsomega}) \left(\frac{p_0}{w_0}\tilde{\bw} - \tilde{\bp}\right)\geq 0\:,
\end{alignedat}
\end{equation}
where the inequality follows from the positive semi-definiteness of $\tilde{H}$ (convexity of $\thA$).

Finally, we prove the strict convexity of $\hA$ provided that $\thA$ is strictly convex. 
Suppose $\thA$ is strictly convex, then its Hessian $\tilde{H}$ is positive definite. Thus the inequality in \eqref{eq:hessianPSD} becomes an equality only when $p_0=0$ and $\frac{p_0}{w_0}\tilde{\bw} - \tilde{\bp} = \mathbf{0}$, which implies that, for all $\bw\in\cR_{\bfm}$, $\bp^T H(\bw) \bp =0$ if and only if $\bp=\mathbf{0}$. Therefore, $H$ is positive definite and $\hA$ is strictly convex.
\end{proof}

To this point, we have shown that a twice differentiable, convex approximation $\hA$ to the Maxwell-Boltzmann entropy $h$ can be extended from a twice differentiable, convex function $\thA$ that approximates $\tilde{h}$, the restriction of $h$ on $\tilde{\cR}_{\bfm}$.
Since $\tilde{\cR}_{\bfm}$ is a bounded set, it is more tractable to generate a dataset that provides a good coverage of $\tilde{\cR}_{\bfm}$. The proposed procedure is summarized as follows.

\renewcommand{\thealgocf}{}

\begin{algorithm}[H]
	\normalsize
	\SetAlgorithmName{Approximation procedure}{} 
	\medskip
	\caption{Steps for constructing $\hA$ from data}
	Sample a normalized dataset $\{(\tilde{\bsomega}^{(i)},\tilde{h}(\tilde{\bsomega}^{(i)}),\tilde{\bsalpha}(\tilde{\bsomega}^{(i)}))\}_{i\in \cI}$ with 
$\tilde{\bsomega}^{(i)} \in\tilde{\cR}_{\bfm}$\;
    Construct a twice differentiable, convex approximation $\thA$ from $\{(\tilde{\bsomega}^{(i)},\tilde{h}(\tilde{\bsomega}^{(i)}),\tilde{\bsalpha}(\tilde{\bsomega}^{(i)}))\}_{i\in \cI}$\;
    Extend $\thA\colon\tilde{\cR}_{\bfm}\to\bbR$ to $\hA\colon\cR_{\bfm}\to\bbR$ by taking	
$\hA(\bw) = w_0 \, \thA(\tilde{\bw}/w_0)  + w_0\log w_0$ for all $\bw\in\cR_{\bfm}$ as given in \eqref{eq:extension}.
This leads to an approximate multiplier function 
\begin{equation}
\label{eq:alpha_scaling}
\halphaA(\bw):=\hA^\prime(\bw) = \left[\begin{array}{c}
\thA(\tilde{\bw}/w_0) + \frac{1}{w_0} \tilde{\bw}^T\tilde{\bsalpha}_{\textup{a}}(\tilde{\bw}/w_0) + \log w_0 + 1\\
\tilde{\bsalpha}_{\textup{a}}(\tilde{\bw}/w_0)
\end{array}\right]\:,
\end{equation}
where $w_0$, $\thA$, and its gradient $\tilde{\bsalpha}_{\textup{a}}$ are needed to compute the gradient of $\hA$\;
\label{algo:extension}
\end{algorithm}	

A flow chart that describes the process of mapping any $\bw\in\cR_{\bfm}$ to the approximate multiplier $\hat{\bsalpha}_{\textup{a}}(\bw)$ using the constructed approximation $\thA$ is included in Figure~\ref{fig:flow_chart}. We next discuss two approaches to construct twice differentiable and convex $\thA$ from data.
\begin{figure}[h]
	\centering
	\includegraphics[width=\linewidth]{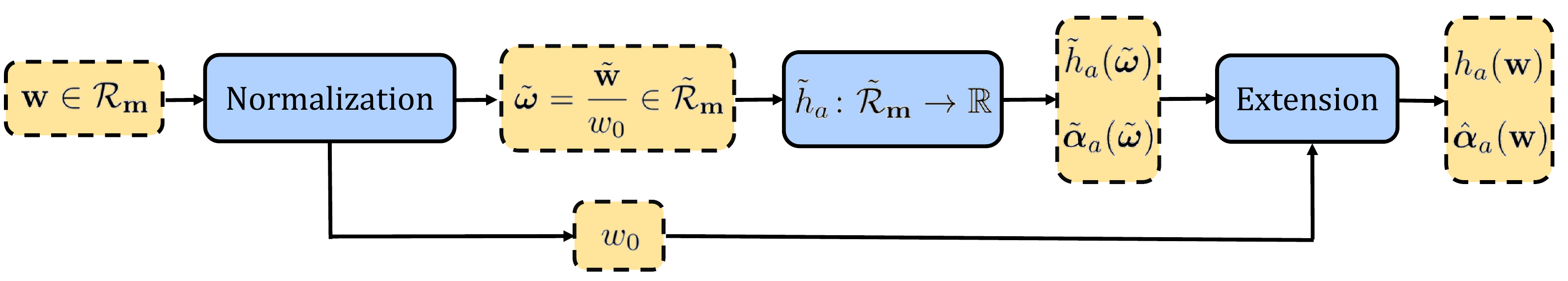}
	\caption{Flow chart: The process of mapping a realizable moment $\bw$ to the approximate multiplier $\hat{\bsalpha}_{\textup{a}}(\bw)$ from the data-driven approximate entropy function $\hA$ is illustrated. The data-driven approximation $\thA$ on the normalized realizable set $\tilde{\cR}_{\bfm}$ is evaluated at the normalized moment $\tilde{\bsomega}$. The result is then extended to the full realizable set $\cR_{\bfm}$ using the formulas given in \eqref{eq:extension} and \eqref{eq:alpha_scaling}.}
	\label{fig:flow_chart}
\end{figure}

\subsection{Data-driven approximations}
\label{subsec:approximations}

In this subsection, we present two approaches to construct twice differentiable and convex approximations to the Maxwell-Boltzmann entropy from the normalized dataset $\{(\tilde{\bsomega}^{(i)},\tilde{h}(\tilde{\bsomega}^{(i)}),\tilde{\bsalpha}(\tilde{\bsomega}^{(i)}))\}_{i\in \cI}$. 
Section~\ref{subsubsec:splines} summarizes the construction of one-dimensional rational cubic spline interpolants proposed in \cite{gregory1986shape} that maintain the convexity of the data.
Section~\ref{subsubsec:NN} describes the procedure for building an approximate entropy using neural networks, which works in arbitrary dimensions.

From here on, we focus on a simplified kinetic model for neutral particle systems with uniform particle traveling speed in a purely scattering medium with slab geometry, which takes the form
\begin{equation}\label{eq:kinetic_eqn_slab}
\p_t f(t,x,\mu) + \mu \p_x f = \sigma_{\textup{s}}(\frac{\vint{f}}{2}-f)\:,
\end{equation}
where the position $x\in\bbR$, angle $\mu\in[-1,1]$, and the scattering crosssection $\sigma_{\textup{s}}>0$ is a constant in $t$, $x$, and $\mu$.
With the basis function $m_i(\mu)$ in $\bfm$ chosen to be the $i$-th order Legendre polynomials, $i=0,\dots,N$, the moment system for \eqref{eq:kinetic_eqn_slab} with approximate entropy $\hA$ can be derived following \eqref{eq:reg_moment_eqn_u}, which gives
\begin{equation}\label{eq:moment_eqn_slab}
\p_t \bu_{\textup{a}} + \p_x \vint{\mu \bfm G_{\halphaA(\bu_{\textup{a}})}} = - \sigma_{\textup{s}} R \bu_{\textup{a}}  \:,
\end{equation}
where $G_{\halphaA(\bw)}= e^{\halphaA(\bw)\cdot\bfm}$ and $\halphaA(\bw):=\hA^\prime(\bw)$ for all $\bw\in\cR_{\bfm}$, and the diagonal matrix $R = \diag(0,1,\dots,1)\in\bbR^{n\times n}$.
This simplification allows for constructing the approximate entropy $\hA$ using the one-dimensional shape-preserving splines in Section~\ref{subsubsec:splines} when the moment order $N=1$. On the other hand, the neural network approach considered in Section~\ref{subsubsec:NN} can approximate the entropy function in multi-dimensions, thus is not limited to the slab geometry case and can work on $N>1$.

\subsubsection{Shape-preserving splines}
\label{subsubsec:splines}

When moment order $N=1$, the moment $\bu(t,x)$ in \eqref{eq:moment_eqn_slab} is a vector in $\bbR^2$ for each $(t,x)$, i.e., $n=2$. 
With $\bfm$ the Legendre polynomial basis, the realizable set becomes $\cR_{\bfm}=\{\bw\in\bbR^2 \colon\,|w_1| < w_0\}$, and the normalized realizable set is simply $\tilde{\cR}_{\bfm}=[-1,1]$. 
The problem then reduces to: finding a twice differentiable, convex function $\thA\colon[-1,1]\to\bbR$ that approximates $\tilde{h}$ given the position, function value, and derivative data $\{(\tilde{\bsomega}^{(i)},\tilde{h}(\tilde{\bsomega}^{(i)}),\tilde{\bsalpha}(\tilde{\bsomega}^{(i)}))\}_{i\in \cI}$, where $\tilde{\bsomega}^{(i)}\in[-1,1]$ for all $i\in\cI$.
The shape-preserving spline interpolation method proposed in \cite{gregory1986shape} constructs a one-dimensional cubic spline interpolant $\thSp(\tilde{\bsomega})$ that is twice differentiable and preserves the monotonicity and convexity of the provided data.
Specifically, for $\{(\tilde{\bsomega}^{(i)},\tilde{h}(\tilde{\bsomega}^{(i)}),\tilde{\bsalpha}(\tilde{\bsomega}^{(i)}))\}_{i\in \cI}$ sampled from a convex function $\tilde{h}$, the cubic spline $\thSp(\tilde{\bsomega})$ is constructed such that, for all $i\in\cI$,
\begin{equation}\label{eq:spline_cond}
\thSp(\tilde{\bsomega}^{(i)}) = \tilde{h}(\tilde{\bsomega}^{(i)}), \quad
\tilde{h}_{\textsf{Sp}}^{\prime}(\tilde{\bsomega}^{(0)}) = \tilde{\bsalpha}(\tilde{\bsomega}^{(0)}), \quad
\tilde{h}_{\textsf{Sp}}^{\prime}(\tilde{\bsomega}^{(|\cI|)}) = \tilde{\bsalpha}(\tilde{\bsomega}^{(|\cI|)}), \quand
\tilde{h}_{\textsf{Sp}}^{\prime\prime}(\tilde{\bsomega}^{(i)}_-) = \tilde{h}_{\textsf{Sp}}^{\prime\prime}(\tilde{\bsomega}^{(i)}_+),
\end{equation}
where the first equality is the interpolation condition, the second and third equalities enforce exact derivative at the two end-points, and the last equality guarantees that $\thSp\in C^2([-1,1])$ with $\tilde{\bsomega}_-$ and $\tilde{\bsomega}_+$ denoting the left and right limit at $\tilde{\bsomega}$, respectively.
To impose convexity, the spline shape parameters are chosen such that 
$
\tilde{h}_{\textsf{Sp}}^{\prime\prime}(\tilde{\bsomega})\geq0 
$
between the interpolation points.
The spline coefficients are then obtained by solving the equations in \eqref{eq:spline_cond} via a Gauss-Seidel approach. See \cite{gregory1986shape} for details.
Once the spline approximation $\tilde{h}_{\textsf{Sp}}$ is constructed, it is straightforward to compute $\tilde{\bsalpha}_{\textsf{Sp}}$ by taking the derivative of $\tilde{h}_{\textsf{Sp}}$.
The approximate multiplier $\hat{\bsalpha}_{\textsf{Sp}}$ can then be reconstructed from $\tilde{h}_{\textsf{Sp}}$ and $\tilde{\bsalpha}_{\textsf{Sp}}$ via \eqref{eq:alpha_scaling}.

While this shape-preserving spline approach gives stable, accurate, and twice differentiable convex approximations to convex functions, its application is limited to one-dimensional functions. 
To the authors' knowledge, there is no existing twice differentiable, convexity-preserving spline approximations in dimensions higher than one.

\subsubsection{Neural networks}
\label{subsubsec:NN}
To handle the more general case that $\tilde{\cR}_{\bfm}$ is in a multi-dimensional space, we use neural networks to approximate the restricted Maxwell-Boltzmann entropy. We denote the neural network approximation as $\thNN(\tilde{\bsomega})$.
For simplicity, we construct $\thNN(\tilde{\bsomega})$ via training standard feed-forward, fully-connected neural networks, i.e., multilayer perceptrons.
Specifically, we consider neural networks of the form
\begin{equation}\label{eq:NN}
\thNN(\tilde{\bsomega};\bstheta) =   (L^{D+1} \circ L^{D} \circ \dots \circ L^{1} \circ L^{0})(\tilde{\bsomega} )\:\text{ with }
L^{\ell}(\bz) =  \sigma(\bsphi^{\ell} \cdot \bz+ \bb^{\ell})\:,\quad
\ell=0,\dots,D+1\:,
\end{equation}
where $\sigma$ denotes the activation function of the network, $\bsphi^\ell$ and $\bb^\ell$ denote the weights and biases at the $\ell$-th layer $L^\ell$, and $\bstheta$ denotes the collection of all parameters $\{(\bsphi^\ell,\bb^\ell)\}_{\ell=0}^D$. 
The number of network parameters depends on the size of the network, e.g., a network of depth $D$ and width $W$ has weights
\begin{equation}
\bsphi^{0} \in \bbR^{(n-1)\times W}, \quad
\bsphi^{\ell} \in \bbR^{W\times W}, \,\ell=1,\dots,D\,,\quand
\bsphi^{D+1} \in \bbR^{W\times 1}, 
\end{equation}
and biases $\bb^{\ell}\in\bbR^{W}$, $\ell=0,\dots,D$ and $\bb^{D+1}\in\bbR$, which sum up to $DW^2+(n+D+1)W+1$ parameters.
In order to guarantee that the network approximation $\thNN$ is twice differentiable, we choose to use the smooth softplus activation function, i.e., $\sigma(\bz) = \log ( 1 + \exp(\bz) )$.
In general, the network function $\thNN$ in \eqref{eq:NN} is not guaranteed to be convex with respect to $\tilde{\bsomega}$. 
In this paper, we verify the convexity of the constructed network approximations by confirming that the Hessian of $\thNN$ is positive semidefinite on a dense test dataset.
To enforce convexity of $\thNN$, one can either penalize negative eigenvalues of the Hessian of $\thNN$ in the training process, or use special network architectures, e.g., \cite{amos2017input}, that guarantee convexity.
The former imposes convexity as a soft constraint on the training dataset and is feasible only when the eigenvalues of Hessian can be obtained efficiently.
The latter gives networks that are convex on the entire input domain, while the special architectures and limitations of the parameters may complicate the training process.

\subsection{Data acquisition strategy}
\label{subsec:sampling_strategy}

The spline and neural network approximations discussed in Section~\ref{subsec:approximations} both rely on some normalized dataset $\{(\tilde{\bsomega}^{(i)},\tilde{h}(\tilde{\bsomega}^{(i)}),\tilde{\bsalpha}(\tilde{\bsomega}^{(i)}))\}_{i\in \cI}$.  
A naive approach to construct the normalized data set is to first sample $\{\tilde{\bsomega}^{(i)}\}_{i\in \cI}$ from $\tilde{\cR}_{\bfm}$, 
evaluate $h$ and $\hat{\bsalpha}$ at $\bw=[1,\tilde{\bsomega}^T]^T$ for each $\tilde{\bsomega}\in\{\tilde{\bsomega}^{(i)}\}_{i\in \cI}$, and then compute $\{\tilde{h}(\tilde{\bsomega}^{(i)})\}_{i\in \cI}$ and $\{\tilde{\bsalpha}(\tilde{\bsomega}^{(i)})\}_{i\in \cI}$ via \eqref{eq:restriction} and \eqref{eq:alpha_scaling_exact}.
This approach is computationally expensive since the evaluation of $h$ and $\hat{\bsalpha}$, i.e., the solution of the optimization problem \eqref{eq:MB_optimization}, is required at each sampled moment. Further, the accuracy of the dataset is limited by the tolerance in the optimization procedure.

To address this issue, we propose to first sample the normalized multipliers $\{\tilde{\bsalpha}^{(i)}\}_{i\in \cI}$ and then compute the associated normalized moments $\{\tilde{\bsomega}^{(i)}\}_{i\in \cI}$ and restricted entropy function values $\{\tilde{h}(\tilde{\bsomega}^{(i)})\}_{i\in \cI}$ from $\{\tilde{\bsalpha}^{(i)}\}_{i\in \cI}$. This approach avoids solving the optimization problem \eqref{eq:MB_optimization} in the data acquisition procedure, and the accuracy of the data is only restricted by the quadrature rule used to perform numerical integrations. 
Specifically, the proposed sampling strategy is as follows.

\begin{algorithm}[H]
	\normalsize
	\SetAlgorithmName{Sampling strategy}{} 
	\medskip
	\caption{Steps for efficient sampling data from $\tilde{\cR}_{\bfm}$}
	Sample the normalized multipliers $\{\tilde{\bsalpha}^{(i)}\}_{i\in \cI}$ from $\bbR^{n-1}$\;
    For every $i\in \cI$, set $\bsalpha^{(i)} = [\alpha_0^{(i)}, (\tilde{\bsalpha}^{(i)})^T]$ with $\alpha_0^{(i)}$ calculated such that $w_0^{(i)}:=\vint{m_0 G_{\bsalpha^{(i)}}} = 1$\;
    For every $i\in \cI$, compute the moment and entropy function value associated to $\bsalpha^{(i)}$ by taking
\begin{equation}	\label{eqn:alpha_to_u}
\bw^{(i)}=\hat{\bw}(\bsalpha^{(i)}) = \vint{\bfm G_{\bsalpha^{(i)}}} \quand
h(\bw^{(i)}) = \vint{\eta(G_{\hat{\bsalpha}(\bw^{(i)})})} = \vint{\eta(G_{\bsalpha^{(i)}})}\:.
\end{equation}
These relations follow from \eqref{eq:primal}--\eqref{eq:entropy} for a general entropy $\eta$\;
Denote $\bw^{(i)}=[w_0^{(i)}, (\tilde{\bw}^{(i)})^T]^T$. Since $w_0^{(i)}=1$, the normalized moment $\tilde{\bsomega}^{(i)} = \tilde{\bw}^{(i)}$ and the restricted entropy function value $\tilde{h}(\tilde{\bsomega}^{(i)})=h(\bw^{(i)})$\;
	\label{algo:sampling}
\end{algorithm}

\section{Numerical results}
\label{sec:num_results}

Numerical results for the proposed data-driven entropy-based closures are presented in this section. 
Section~\ref{subsec:implementation} gives the implementation details, including the sampling of datasets and the training setting for the neural networks.
Section~\ref{subsec:closure_summary} summarizes the constructed closures and their approximation accuracy.
Section~\ref{subsec:realizable_boundary_test} compares the computation time between the neural network closures and the standard optimization approach for Maxwell-Boltzmann entropy-based closures.
Finally, the data-driven closures are applied to solve a plane source benchmark problem, and the results are reported in Section~\ref{subsec:planesource}.

\subsection{Implementation details}
\label{subsec:implementation}

\subsubsection{Sampled data}
\label{subsubsec:data}

We follow the approach outlined in Section~\ref{subsec:sampling_strategy} to sample the data needed in the construction of data-driven closures considered in this section. With the proposed approach, we only need to determine a strategy for sampling the normalized multipliers $\{\tilde{\bsalpha}^{(i)}\}_{i\in \cI}$, and the remaining parts of the data can be computed accordingly.
In the tests considered in this section, we sample each entry of $\tilde{\bsalpha}$ on a uniform grid of size $s$ in a prescribed interval $[\tilde{\alpha}_{\min},\tilde{\alpha}_{\max}]$, e.g., when $\tilde{\bsalpha}\in\bbR^{n-1}$, then each sample $\tilde{\bsalpha}^{(i)}$ is a point on the uniform tensor mesh of size $s^{n-1}$ in the hypercube $[\tilde{\alpha}_{\min},\tilde{\alpha}_{\max}]^{n-1}$. 
With $\tilde{\alpha}_{\min}<0$ sufficiently small and $\tilde{\alpha}_{\max}>0$ sufficiently large, this sampling strategy allocates more data points near the boundary of the normalized realizable set $\tilde{\cR}_{\bfm}$ for the Maxwell-Boltzmann entropy $h_{\MB}$, which significantly improves the quality of data-driven closures since $h_{\MB}$ is known to vary drastically near the boundary of $\tilde{\cR}_{\bfm}$.

We choose $[\tilde{\alpha}_{\min},\tilde{\alpha}_{\max}]=[-65, 65]$ when training all data-driven closures for the case $N=1$. This choice covers normalized moments in $[-1+0.015,1-0.015]$, whereas the normalized realizable set $\tilde{\cR}_{\bfm}=[-1,1]$.
For the spline closures, the data size $s$ is equivalent to the size of the spline approximation, i.e., the number of interpolation points.
For the network closures, we choose the data size $s=10,000$ for training networks with various depths and widths.
As for training neural network closures for the case $N=2$, we choose $[\tilde{\alpha}_{\min},\tilde{\alpha}_{\max}]=[-10, 10]$ with $100$ grid points in the $\tilde{\alpha}_1$ direction and $50$ grid points in the $\tilde{\alpha}_2$ direction.

In Section~\ref{subsec:closure_summary}, we report the training and test errors of data-driven closures constructed with the training datasets described above. The test datasets are sampled using the same approach, but with a denser uniform grid for each entry of $\tilde{\bsalpha}$.
In addition, to test the accuracy of the reconstruction steps in \eqref{eq:extension} and \eqref{eq:alpha_scaling}, we allow the zero-th moment $w_0$ to vary in the test datasets while $w_0$ in the training datasets is always normalized to one.
When $N=1$, the spline and network approximations are both tested on a dataset with 160 $w_0$ values evenly distributed on interval $[10^{-8}, 8]$.
For each $w_0$ value, $\tilde{\bsalpha}$ is sampled on a uniform grid of 52,000 points with $[\tilde{\alpha}_{\min},\tilde{\alpha}_{\max}]=[-65, 65]$.
When $N=2$, the test dataset for network approximations also has 160 evenly distributed $w_0$ in $[10^{-8}, 8]$.
For each $w_0$, we choose $[\tilde{\alpha}_{\min},\tilde{\alpha}_{\max}]=[-10, 10]$ and sample $\tilde{\bsalpha}$ on a uniform tensor grid with 200 grid points in both the $\tilde{\alpha}_1$ and $\tilde{\alpha}_2$ directions.

\subsubsection{Neural network training}
\label{subsubsec:training}
The neural network approximations are implemented as Keras models with tensorflow~2.0 in Python.

\paragraph{Network initialization}  
In the neural network training procedure, we initialize the network parameters using a strategy considered in \cite{kumar2017weight}, which aims to normalize the variance of each network layer output with respect to a standard normal random variable input to one. 
This strategy is motivated from the common Kaiming initialization \cite{he2015delving} for neural networks with the Rectified Linear Unit (ReLU) activation function.
We adopt the initialization strategy and apply it on the softplus activation function, $\sigma(\bz) = \log ( 1 + \exp(\bz) )$, which is used here to guarantee the regularity of network approximations.
Specifically, for a network of depth $D$ and width $W$, we initialize all biases to be zero, i.e., $\bb^{\ell} = 0$, $\ell= 0,\dots, D+1 $, and draw the initial weights $\bstheta^{\ell}$ from $\cN(0,s^2_\ell)$, the zero-mean normal distribution with variance $s^2_\ell$, where the variance for each layer is given by
\begin{equation}
s^2_{0} = \frac{1}{(n-1)  \sigma'(0)^2  (1+\sigma(0)^2) }, \quand
s^2_{\ell} = \frac{1}{W  \sigma'(0)^2  (1+ \sigma(0)^2) }, \quad  \ell= 1,\dots, D+1 \:.
\end{equation}

\paragraph{Training and validation data}
We follow the standard practice and split the sampled data in Section~\ref{subsubsec:data} into the training and validation datasets. 
In each test, we uniformly draw 10\% of the sampled data to be the validation data, i.e., a 90--10 split.
From here on, we denote the index sets of the training and validation datasets as $\cI^{\train}$ and $\cI^{\validation}$, respectively.

\paragraph{Network output --- multiplier and symmetry}
The network closure aims to provide an approximate multiplier $\hat{\bsalpha}_{\textsf{NN}}(\bw)$ at a given realizable moment $\bw\in\cR_{\bfm}$. It follows from \eqref{eq:alpha_scaling} that both the network function value $\thNN$ and the network gradient $\tilde{\bsalpha}_{\textsf{NN}}$ are needed to compute $\hat{\bsalpha}_{\textsf{NN}}(\bw)$.
To obtain the gradient $\tilde{\bsalpha}_{\textsf{NN}}$ from the fully-connected network $\thNN$ discussed in Section~\ref{subsubsec:NN}, we append at the end of the fully-connected network a custom layer that performs automatic differentiation by calling \texttt{GradientTape.gradient()} in tensorflow.
With the choice of Legendre polynomials as moment basis $\bfm$, the entropy function $h$ is known to be symmetric with respect to the moments of odd orders, e.g., $h(\bw)=h(\bw^*)$ if $w_i^*=w_i$ for even $i$ and $w_i^*=-w_i$ for odd $i$.
To impose this symmetry in the neural network approximations, we define
\begin{equation}\label{eq:symmetry}
{\thNN}^{\text{sym}}(\tilde{\bsomega}) := \frac{1}{2} ({\thNN}(\tilde{\bsomega}) + {\thNN}(\tilde{\bsomega}^*))\quad\text{for all } \tilde{\bsomega} \in \tilde{\cR}_{\bfm}\:,
\end{equation}
where $\tilde{\omega}_i^*=\tilde{\omega}_i$ for even $i$ and $\tilde{\omega}_i^*=-\tilde{\omega}_i$ for odd $i$.
This symmetric network approximation is used in the plane source tests discussed in Section~\ref{subsec:planesource}.

\paragraph{Loss function} 
The training procedure aims to find the optimal network parameter $\bstheta$ by minimizing a loss function. The loss function considered here takes the form
\begin{equation}\label{eq:loss_function}
L(\bstheta;\cI) = E_h^2(\bstheta;\cI) + \lambda(\cI) E_{\bw}^2(\bstheta;\cI)\:,
\end{equation}
where $\lambda(\cI) = (\sum_{i\in\cI}{\| \bw^{(i)} \|^{2}})^{-1} $ with $\bw^{(i)} = [1,(\tilde{\bsomega}^{(i)})^T]^T$ and the errors are defined as
\begin{equation}\label{eq:error_def}
E_h^2(\bstheta;\cI):= \sum_{i \in \cI} | \thNN(\tilde{\bsomega}^{(i)};\bstheta) - \tilde{h}(\tilde{\bsomega}^{(i)})|^2, \quad 
E_{\bw}^2(\bstheta;\cI):=\sum_{i \in \cI} \norm{ {\bw}_{\textsf{NN}}(\tilde{\bsomega}^{(i)};\bstheta) - \bw^{(i)}}^{2}.
\end{equation}
Here ${\bw}_{\textsf{NN}}(\tilde{\bsomega}^{(i)};\bstheta)$ is the approximate moment in $\cR_{\bfm}$ given by the neural network closure, i.e., ${\bw}_{\textsf{NN}}(\tilde{\bsomega}^{(i)};\bstheta):=\vint{\exp(\bfm\cdot\hat{\bsalpha}_{\textsf{NN}}(\tilde{\bsomega}^{(i)};\bstheta)}$ with $\hat{\bsalpha}_{\textsf{NN}}$ the approximate multiplier that can be computed using $\thNN$ and $\tilde{\bsalpha}_{\textsf{NN}}$ via \eqref{eq:alpha_scaling}.
We choose to include the moment reconstruction error $E_{\bw}$ in the loss function rather than the multiplier error
\begin{equation}
E_{\bsalpha}^{2}(\bstheta;\cI):=\sum_{i \in \cI} \| \hat{\bsalpha}_{\textsf{NN}}(\tilde{\bsomega}^{(i)};\bstheta) - \hat{\bsalpha}^{(i)}\|^2,
\end{equation}
because, in the moment system \eqref{eq:reg_moment_eqn_u}, the closures are used to evaluate the flux term and the collision term, and both of which are integrated quantities over the velocity/angular space. Thus we conjecture that the moment reconstruction error $E_{\bw}$ would be a more reliable measure for the closure accuracy than the multiplier error $E_{\bsalpha}$. We will verify this conjecture in the plane source tests reported in Section~\ref{subsec:planesource}.

\paragraph{Hyperparameters and optimizer} 
The networks were scheduled to be trained over 15,000 epochs with a batch size 50. The training process terminates early if one of the following two criteria is satisfied: (i) the validation error $E_{\bw}^2(\bstheta,\cI^{\validation})$ is less than tolerance $10^{-8}$, and (ii) there is no improvement greater than $10^{-9}$ on $E_{\bw}^2(\bstheta,\cI^{\validation})$ in the past 1,500 epochs.
To minimize the loss function \eqref{eq:loss_function}, the Adam optimizer \cite{kingma2014adam}, a first-order stochastic gradient descent algorithm, is used with an adaptive learning rate $r = 10^{-3+(m/5000)}$, where $m$ is the counter for the current epoch.

\subsection{Constructed data-driven closures}
\label{subsec:closure_summary}

Here we list several data-driven closures constructed using the training datasets specified in Section~\ref{subsubsec:data}. The neural network closures are trained with the specifics detailed in Section~\ref{subsubsec:training} and the spline closures are built with the procedure described in Section~\ref{subsubsec:splines}, where a Gauss-Seidel approach is used to solve for the spline coefficients up to a $10^{-12}$ relative tolerance.
From here on, we denote the neural network closures of various sizes as $\texttt{NN}_N^{D\times W}$, where $D$ and $W$ are respectively the depth and width of the network, and $N$ is the moment order. The spline closures are denoted as $\texttt{S}^P$, where $P$ is the number of interpolation points in the spline.
Tables~\ref{table:constructed_summary_by_error_M1} and \ref{table:constructed_summary_by_size_M1} report the training and test errors of various spline and neural network closures for moment order $N=1$. In both tables, the sizes are the number of parameters in the closures. For neural network closures, the formula is given in Section~\ref{subsubsec:NN}. For spline closures, the sizes are identical to the number of interpolation points.
Here the reported training and test errors for the entropy function values, moments, and multipliers, are the relative root-mean-square errors (RMSEs), e.g., the training and test errors in function value are defined respectively as
\begin{equation}
\textstyle
\Err_{h}^{\train} = (\sum_{i\in\cI^{\train}}|h_{\textup{a}}(\tilde{\bsomega}^{(i)})|^2)^{-\frac{1}{2}} E_{h}(\bstheta,\cI^{\train})\:,\quand
\Err_{h}^{\test} = (\sum_{i\in\cI^{\test}}|h_{\textup{a}}(\tilde{\bsomega}^{(i)})|^2)^{-\frac{1}{2}} E_{h}(\bstheta,\cI^{\test})\:,
\end{equation}
where $E_{h}$ is the RMSE as defined in the loss function \eqref{eq:loss_function}. 
The test and training errors for the moments and multipliers are all defined analogously.
Here the test datasets for $N=1$ and $2$ are also specified in Section~\ref{subsubsec:data}.

Table~\ref{table:constructed_summary_by_error_M1} compares the spline and neural network closures that have testing moment errors ($\Err_{\bw}^{\test}$) around $10^{-2}$, $10^{-3}$, and $10^{-4}$. Table~\ref{table:constructed_summary_by_size_M1} compares the accuracy of the spline and neural network closures that are of similar sizes (around 100, 500, and 1000). 
The results reported in these two tables suggest that neural network approximations generally need many more parameters to achieve similar accuracy as the spline approximations. When the two approximations are of similar size, the test errors of the spline approximation are at least two orders of magnitude lower than the ones of the network approximation, and the different increases as the number of parameter grows. We note that, for spline approximations, $\Err_h^{\train}$ is zero by construction while $\Err_{\bsalpha}^{\train}$ is nonzero since the spline interpolation requires derivative to be exact only at the end points.

\begin{table}[h]
	\centering 
	\begin{tabular}{|l|c|c|c|c|c|c|}
		\toprule
		Closure & $\texttt{NN}^{1\times15}_{1}$ & $\texttt{NN}^{3\times15}_{1}$ & $\texttt{NN}^{5\times30}_{1}$ & $\texttt{S}^{30}$ & $\texttt{S}^{60}$ & $\texttt{S}^{130}$ \\
		\hline
		Size         				      &             286 &             766 &            4741 &        30 &        60 &        130 \\ \hline
		$\Err_h^{\train}$				&        6.28e-3 &        5.11e-4 &        1.14e-4 &  \,\,\,\,\,0e0 &  \,\,\,\,\,0e0 & \,\,\,\,\,0e0 \\
		$\Err_{\bw}^{\train}$           &        7.31e-3 &        7.70e-4 &        1.93e-4 &  5.61e-3 &  4.52e-4 &   1.91e-5 \\
		$\Err_{\hat{\bsalpha}}^{\train}$  &        1.00e-1 &        1.62e-2 &        4.36e-3 &  1.03e-3 &  7.02e-5 &   3.29e-6 \\
		$\Err_h^{\test}$               &        3.03e-3 &        2.46e-4 &        5.49e-5 &  4.30e-3 &  2.28e-4 &   9.48e-6 \\
		$\Err_{\bw}^{\test}$            &        7.31e-3 &        7.69e-4 &        1.92e-4 &  1.06e-2 &  9.36e-4 &   7.86e-5 \\
		$\Err_{\hat{\bsalpha}}^{\test}$   &        1.01e-1 &        1.64e-2 &        4.35e-3 &  1.11e-3 &  9.69e-5 &   8.67e-6 \\
		\bottomrule
	\end{tabular}
 	\caption{Training and test errors in terms of entropy function values ($h$), moments ($\bw$), and multipliers ($\hat{\bsalpha}$) for neural network and spline approximations in the case $N=1$ with testing moment errors ($\Err_{\bw}^{\test}$) around $10^{-2}$, $10^{-3}$, and $10^{-4}$.} 
	\label{table:constructed_summary_by_error_M1}
\end{table}

\begin{table}[h]
	\centering 
	\begin{tabular}{|l|c|c|c|c|c|c|}
		\toprule
		Closure & $\texttt{NN}^{2\times15}_1$ & $\texttt{NN}^{0\times45}_1$ & $\texttt{NN}^{4\times15}_1$ & $\texttt{S}^{100}$ & $\texttt{S}^{500}$ & $\texttt{S}^{1000}$ \\
		\hline
		Size         &             136 &             526 &            1006 &        100 &        500 &        1000 \\ \hline		
		$\Err_h^{\train}$      &        8.48e-2 &        1.76e-3 &        3.77e-4 &   \,\,\,\,\,0e0 &   \,\,\,\,\,0e0 &    \,\,\,\,\,0e0 \,\,\\
		$\Err_{\bw}^{\train}$      &        7.14e-2 &        2.43e-3 &        5.87e-4 &   5.47e-5 &   8.78e-8 &    5.50e-9 \,\,\\
		$\Err_{\hat{\bsalpha}}^{\train}$  &        4.71e-1 &        4.07e-2 &        1.18e-2 &   9.40e-6 &   1.53e-8 &    9.61e-10 \\
		$\Err_h^{\test}$     &        4.12e-2 &        8.48e-4 &        1.82e-4 &   2.75e-5 &   4.42e-8 &    2.79e-9 \,\,\\
		$\Err_{\bw}^{\test}$     &        7.14e-2 &        2.42e-3 &        5.88e-4 &   1.80e-4 &   1.32e-6 &    1.64e-7 \,\,\\
		$\Err_{\hat{\bsalpha}}^{\test}$ &        4.77e-1 &        4.10e-2 &        1.20e-2 &   1.96e-5 &   1.47e-7 &    1.85e-8 \,\,\\
		\bottomrule
	\end{tabular}
	\caption{Training and test errors in terms of entropy function values ($h$), moments ($\bw$), and multipliers ($\hat{\bsalpha}$) for neural network and spline approximations of size around 100, 500, and 1000 in the case $N=1$.} 
	\label{table:constructed_summary_by_size_M1}
\end{table}

To further explore the relation between the test errors and numbers of parameters for the spline and neural network approximations, we have trained and tested more approximations of different sizes and reported their test errors in Figure~\ref{fig:M1_RMSE_by_Size}.
It can be observed from Figure~\ref{fig:M1_RMSE_by_Size} that (i) the spline approximation is indeed third-order accurate; (ii) the networks generally require substantially more parameters than the splines to reach the same level of accuracy; and (iii) the network depth is more important than its width in this test, i.e., for networks of similar size, deep but narrow networks outperform wide but shallow ones.
We also note that $\Err_{\bsalpha}$ is often orders of magnitude higher than $\Err_{h}$ and $\Err_{\bw}$. This is because $\hat{\bsalpha}$ becomes unbounded at the realizable boundary, which makes the approximation difficult.

\begin{figure}[h]
	\centering 
	\subfigure[Entropy function error $\Err_{h}^{\test}$]{\includegraphics[width = 0.28\textwidth]{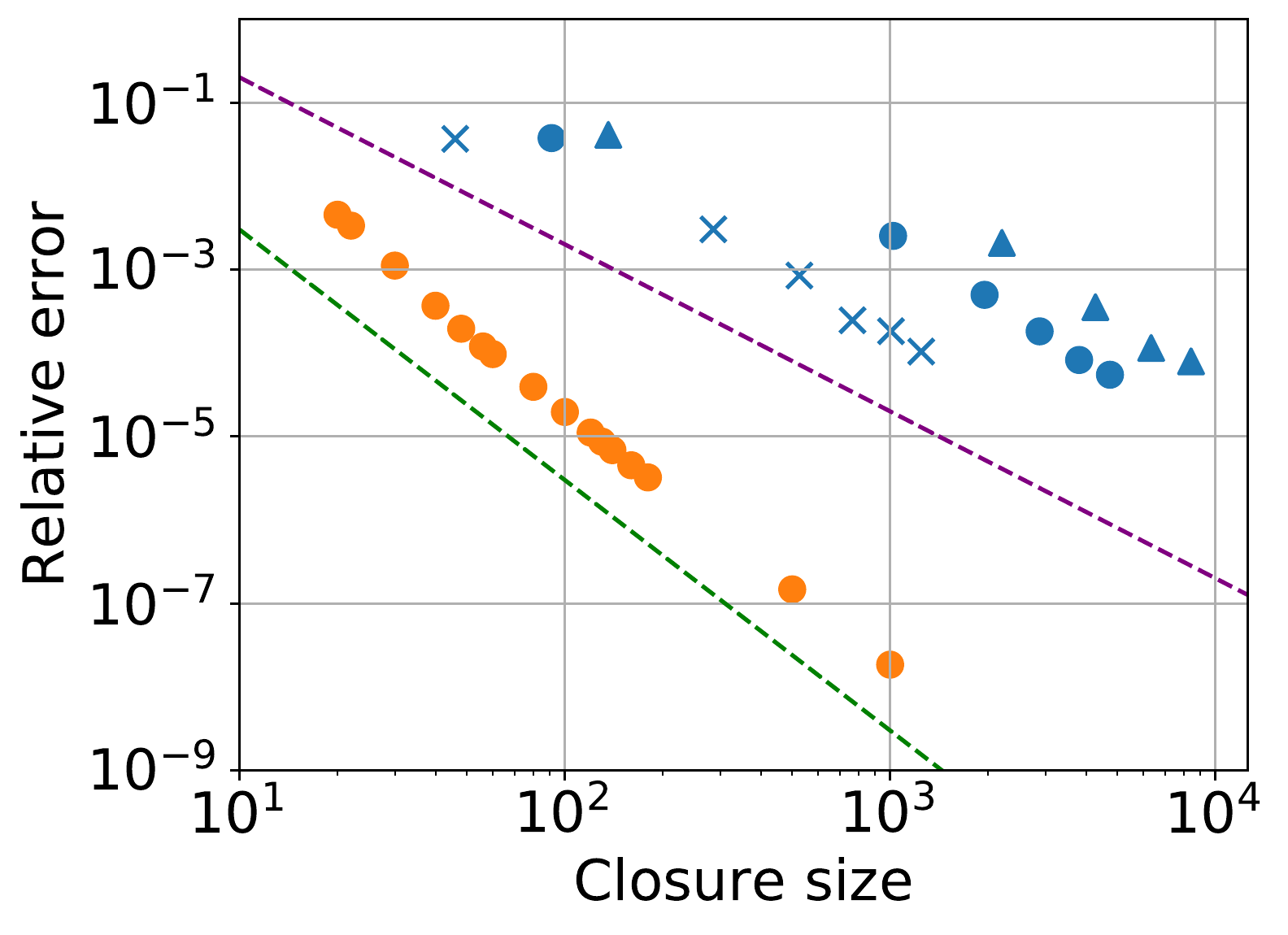}}~~
	\subfigure[Moment error $\Err_{\bw}^{\test}$]{\includegraphics[width= 0.28\textwidth]{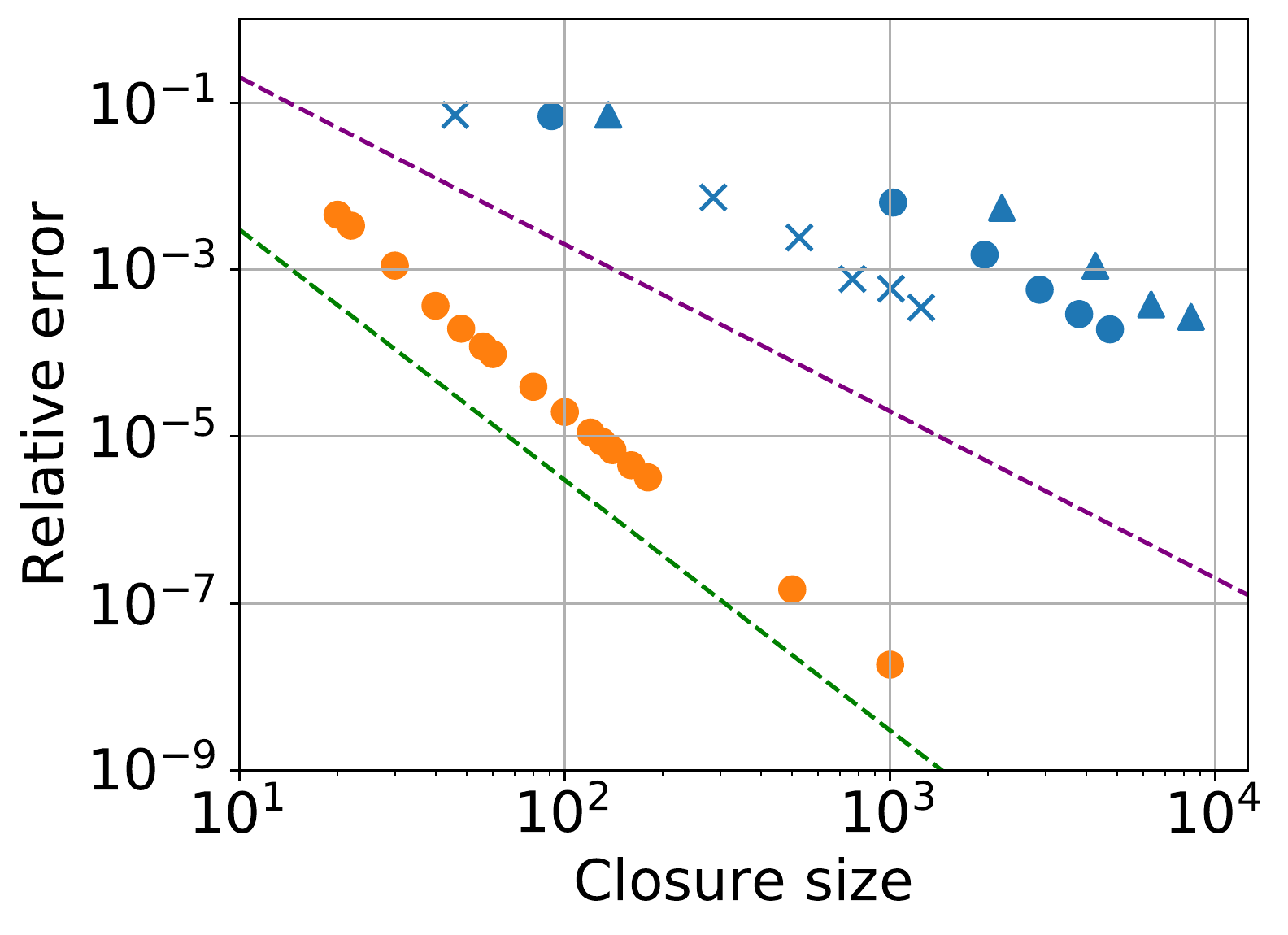}}~~
	\subfigure[Multiplier error $\Err_{\bsalpha}^{\test}$]{\includegraphics[width = 0.408\textwidth]{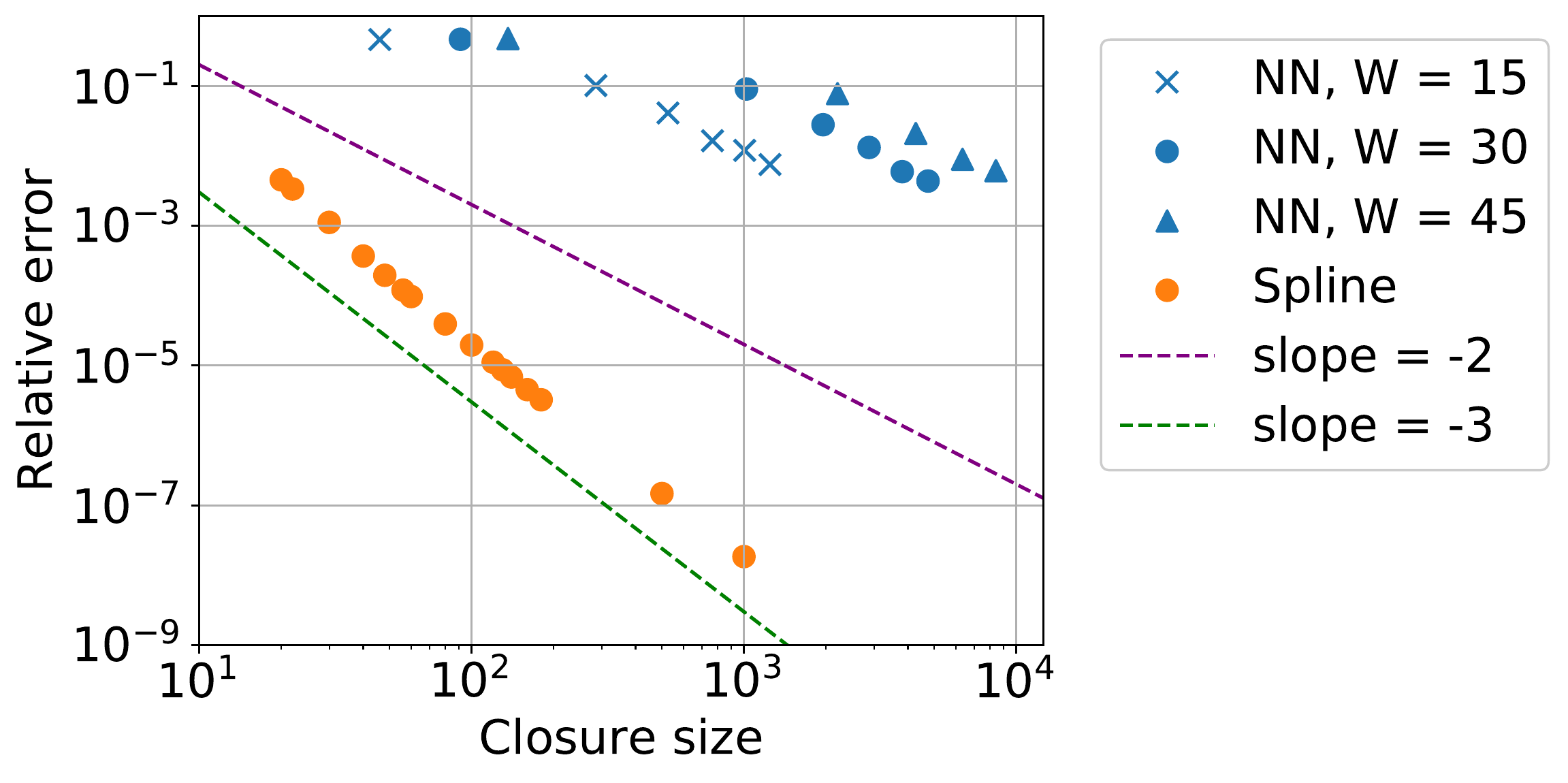}}
	\caption{Relative test errors for the constructed data-driven closures for $N = 1$. Neural network closures of various depths with width $W = $ 15, 30, and 45 are denoted as blue crosses, circles, and triangles. Spline closures of various sizes are denoted as orange circles.   }
	\label{fig:M1_RMSE_by_Size}
\end{figure}

Table~\ref{table:M2ErrorSummary} shows three neural network approximations for the case $N=2$ that have relative moment reconstruction (test) errors around $10^{-2}$, $10^{-3}$, and $10^{-4}$. Here the spline closures are not included since the convex spline approximation in Section~\ref{subsubsec:splines} is restricted to one dimension and the normalized realizable set $\tilde{\cR}_{\bfm}$ is in $\bbR^2$ when $N=2$.
The results in Table~\ref{table:M2ErrorSummary} indicate that more parameters are needed in the network approximations for $N=2$ to achieve similar levels of accuracy compared to the network approximations for $N=1$, which is as expected since the dimension of the domain of approximation functions grows as $N$ increases.
\begin{table}[h]
	\begin{center}
		\begin{tabular}{|l|c|c|c|}
			\toprule
			Closures & $\texttt{NN}^{1\times15}_2$ & $\texttt{NN}^{3\times30}_2$ & $\texttt{NN}^{4\times45}_2$ \\
			\hline
			Size         &         301 &        2911 &        8461 \\			\hline
			$\Err_h^{\train}$      &    1.89e-2 &    4.73e-4 &    3.24e-4 \\
			$\Err_{\bw}^{\train}$     &    1.74e-2 &    8.13e-4 &    5.67e-4 \\
			$\Err_{\bsalpha}^{\train}$  &    3.98e-1 &    2.25e-1 &    2.17e-1 \\
			$\Err_h^{\test}$     &    9.02e-3 &    2.34e-4 &    1.55e-4 \\
			$\Err_{\bw}^{\test}$     &    1.76e-2 &    8.08e-4 &    5.63e-4 \\
			$\Err_{\bsalpha}^{\test}$ &    4.05e-1 &    2.31e-1 &    2.21e-1 \\
			\bottomrule
		\end{tabular}
	\end{center}
 	\caption{Training and test errors in terms of entropy function values ($h$), moments ($\bw$), and multipliers ($\hat{\bsalpha}$) for neural network approximations in the case $N=2$ with testing moment errors ($\Err_{\bw}^{\test}$) around $10^{-2}$, $10^{-3}$, and $10^{-4}$.} 
	\label{table:M2ErrorSummary}
\end{table}

Figure~\ref{fig:M2NetSizeError} reports the test errors of more neural network approximations of various width and depth for $N=2$. 
The results are similar to the ones in Figure~\ref{fig:M1_RMSE_by_Size} for the $N=1$ case, and we also observe that, to improve the network accuracy, it is preferable to increase the network depth than the width.
\begin{figure}[h]
	\centering 
	\subfigure[Entropy function error $\Err_{h}^{\test}$]{\includegraphics[width= 0.285\textwidth]{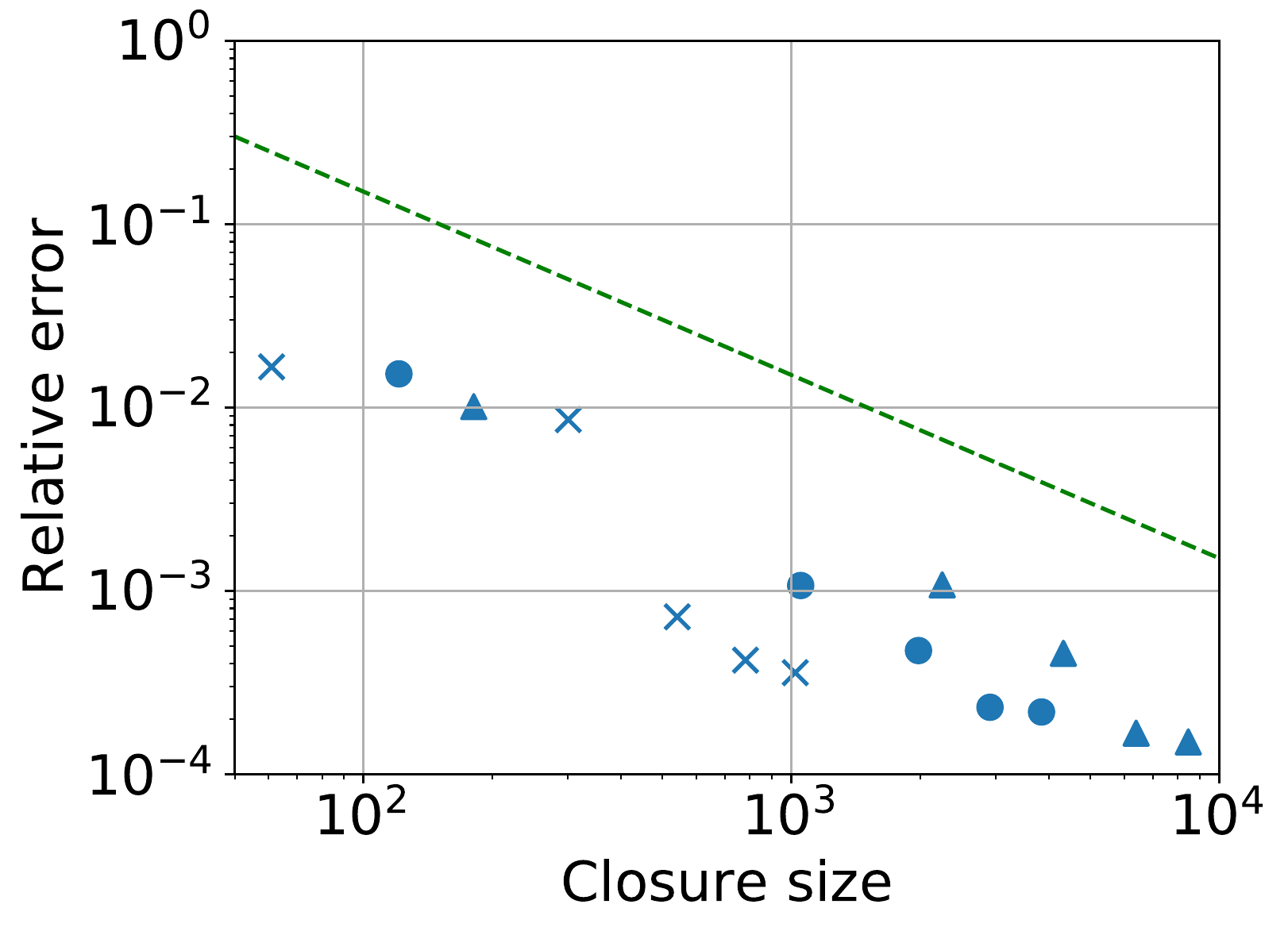}}~~
	\subfigure[Moment error $\Err_{\bw}^{\test}$]{\includegraphics[width = 0.285\textwidth]{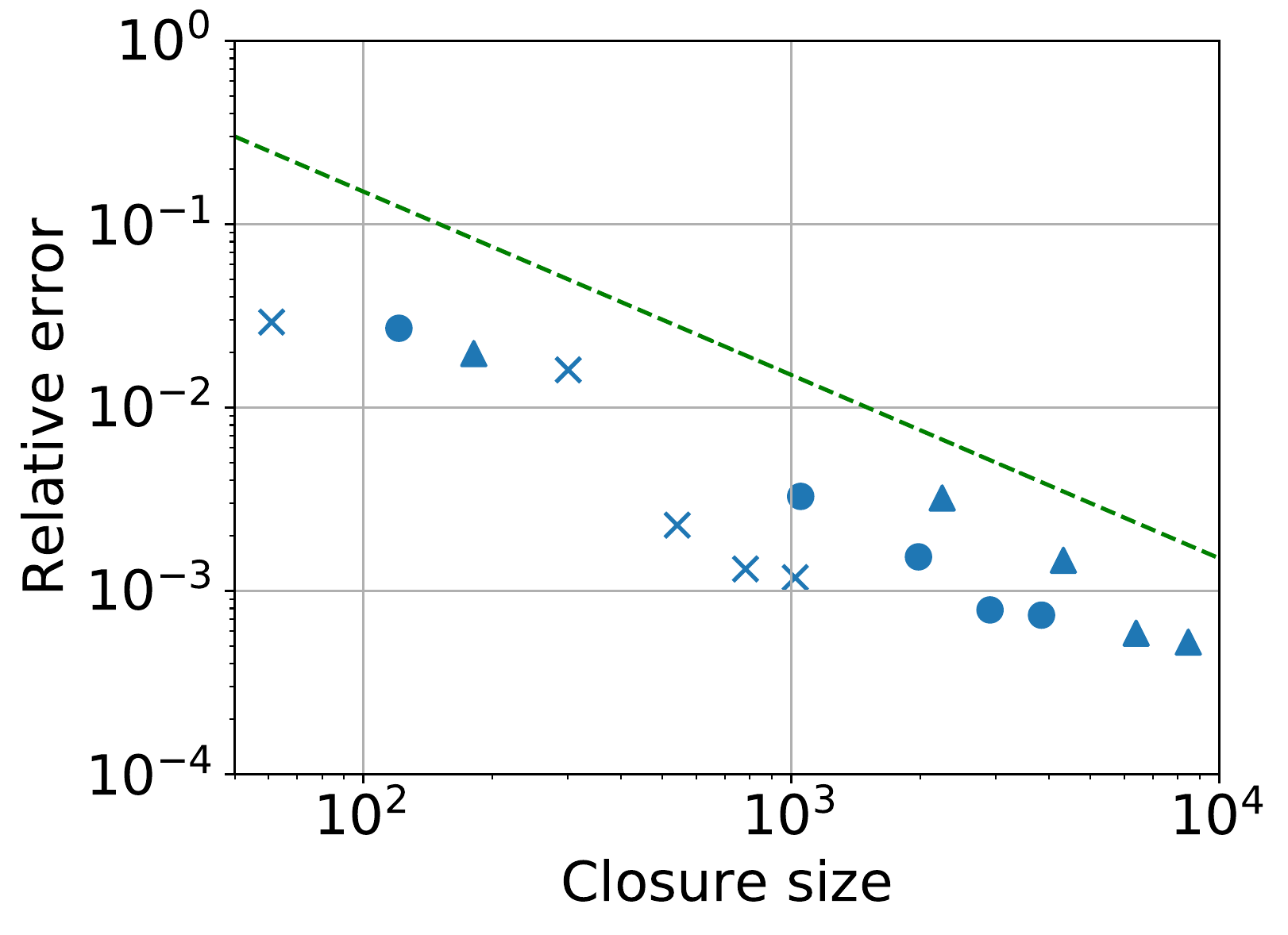}}~~
	\subfigure[Multiplier error $\Err_{\bsalpha}^{\test}$]{\includegraphics[width = 0.407\textwidth]{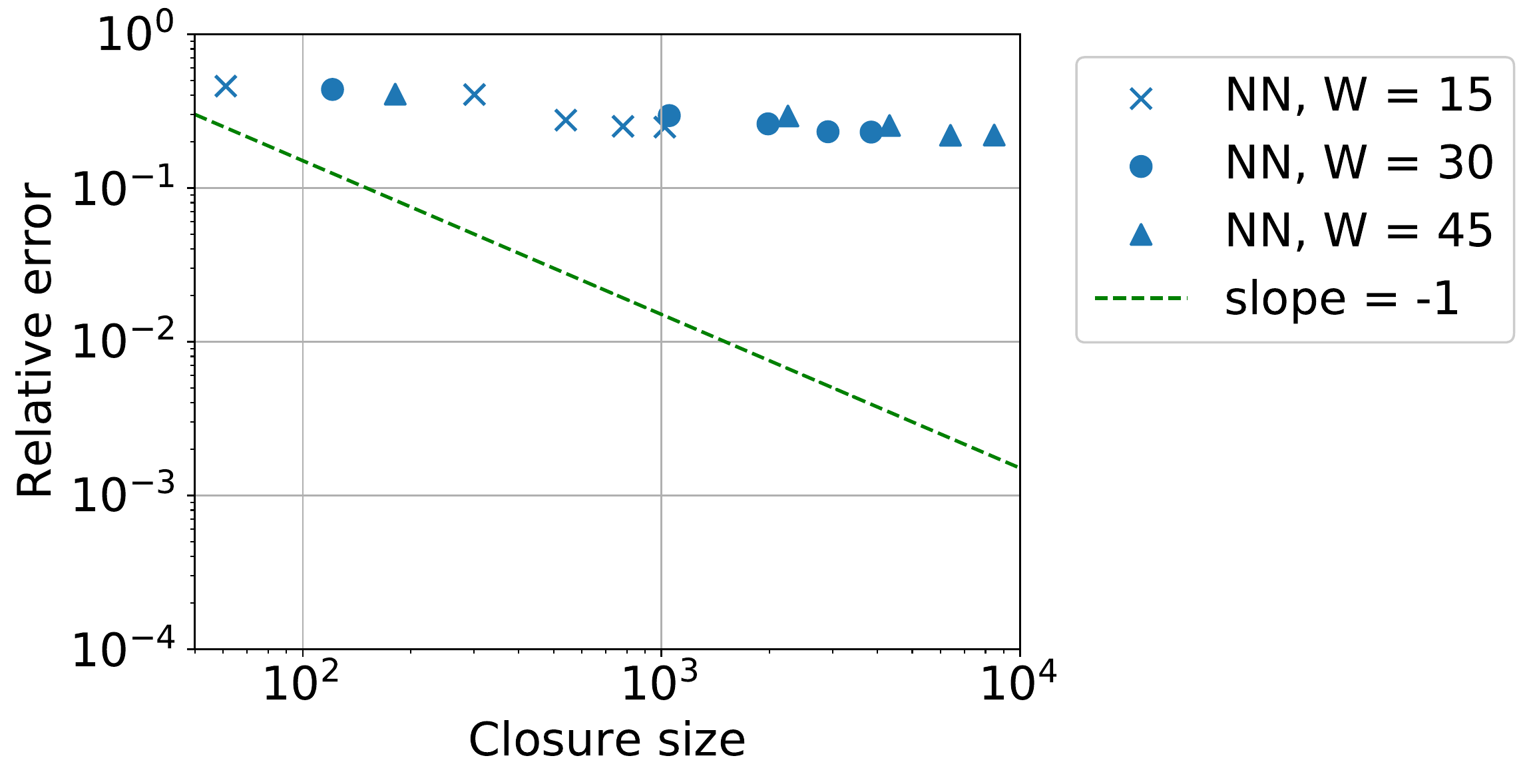}}
	\caption{Relative test errors for the constructed neural network closures for $N = 2$. Closures of various depths with width $W = $ 15, 30, and 45 are denoted as blue crosses, circles, and triangles.}
	\label{fig:M2NetSizeError}
\end{figure}

The data-driven closures considered in this section are all twice differentiable by construction. The spline approximations are guaranteed to be convex. On the other hand, the convexity of the network approximations is verified by checking the positive semidefiniteness of the Hessian matrix at each point of the excessively refined test datasets given in Section~\ref{subsubsec:data}. For both the $N=1$ and $N=2$ cases, we do not observe any violation of the convexity condition on the test datasets in the network approximations considered here.

\subsection{Computation time comparison}
\label{subsec:realizable_boundary_test}

In this section, we compare the computation time needed for the network and optimization approaches to map given moments $\bw$ to the associated multipliers $\hat{\bsalpha}(\bw)$. The network approach maps $\bw$ to $\hat{\bsalpha}(\bw)$ via a direct evaluation, whereas the optimization approach generates $\hat{\bsalpha}(\bw)$ by solving \eqref{eq:MB_optimization}. The optimization algorithm compared here is a Newton-based solver with backtracking line search considered in \cite{alldredge2012high}. We set the optimization tolerance to be $10^{-8}$ throughout the tests in this section. 
It is known \cite{alldredge2012high,alldredge2014adaptive,alldredge2019regularized} that, for moments close to the realizable boundary, the Hessian in \eqref{eq:MB_optimization} becomes ill-conditioned, which leads to more iteration counts in the optimization. 
Here we will confirm this result in the cases when the moment order $N=1$ and $2$, and verify that the cost of evaluating neural network approximations remains roughly constant.

We first compare the network and optimization approaches in the $N=1$ case on a set of moments $\{\bw^{(j)}\}_{j=1}^{200}$ with $\bw^{(j)}:=[1,w_1^{(j)}]$ and $w_1^{(j)}$ evenly distributed in $[0,1]$. 
Here we consider two variants of the optimization approach, $\optQ$ and $\optA$, where  $\optQ$ performs $\vint{\cdot}$, the angular integration in \eqref{eq:MB_optimization}, using a 30-point Gauss-Legendre quadrature, and $\optA$ uses integration based on an analytic formula, which is available for the special case $N=1$.
The network evaluation is performed using Keras function \texttt{model.predict()}.
Figure~\ref{fig:M1_boundaryline_test_iter} shows the iteration counts needed for the optimization algorithm to converge and Figure~\ref{fig:M1_boundaryline_time} shows the computation time for both the optimization and network evaluation. In these results, the computation is restricted to one CPU core, and the network is evaluated at one sample of moments at a time, i.e., \texttt{model.predict()} is called 200 times to evaluate the network at $\{\bw^{(j)}\}_{j=1}^{200}$, with a single point $\bw^{(j)}$ evaluated each time.
Figure~\ref{fig:M1_boundaryline_NNtime} reports the median of network evaluation time as well as the first and third quartile values over 20 trials.
The results in Figure~\ref{fig:M1_boundaryline_test} confirm that the iteration counts for both $\optQ$ and $\optA$ grow as the moments gets closer to the realizable boundary and the increased iteration counts are reflected in the computation time.
It can also be observed from the results that the single-point neural network evaluation is significantly slower than both variants of the optimization approach and that evaluating a smaller network ($\texttt{NN}^{0\times15}_1$) is only slightly faster than evaluating a larger one ($\texttt{NN}^{4\times45}_1$) at a single point. We suspect that the overhead incurred when calling the Keras function \texttt{model.predict()} may dominant the computation time in network evaluations. In order to minimize the overhead, we next use \texttt{model.predict()} to perform batch evaluation and compare the result to the optimization time in the next test.

\begin{figure}[h]
	\centering 
	\subfigure[Iteration count]{\includegraphics[width = 0.265\linewidth]{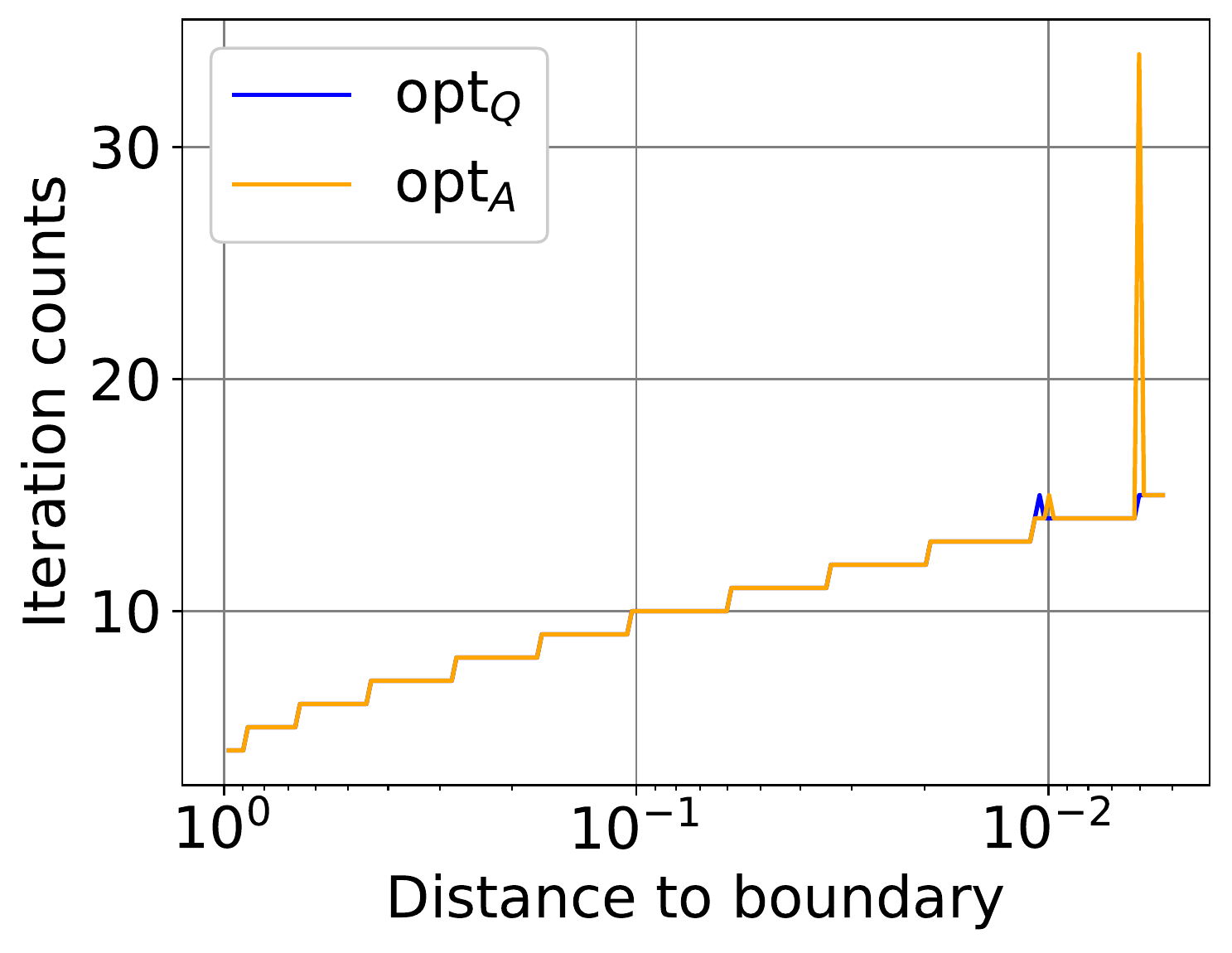} \label{fig:M1_boundaryline_test_iter}}
	\subfigure[Compuation time]{\includegraphics[width = 0.278\linewidth]{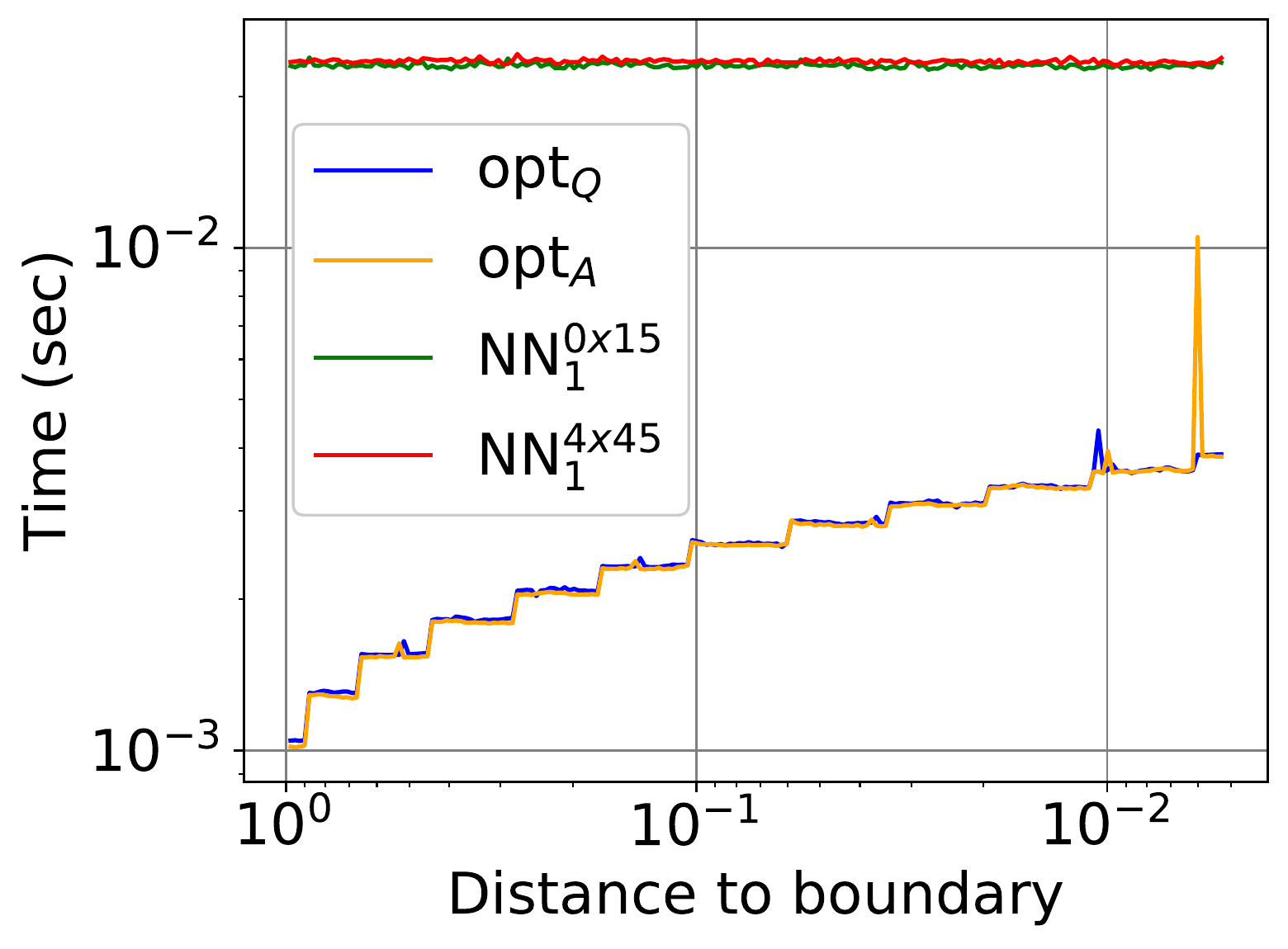} \label{fig:M1_boundaryline_time}}
	\subfigure[Compuation time (networks)]{\includegraphics[width = 0.41\linewidth]{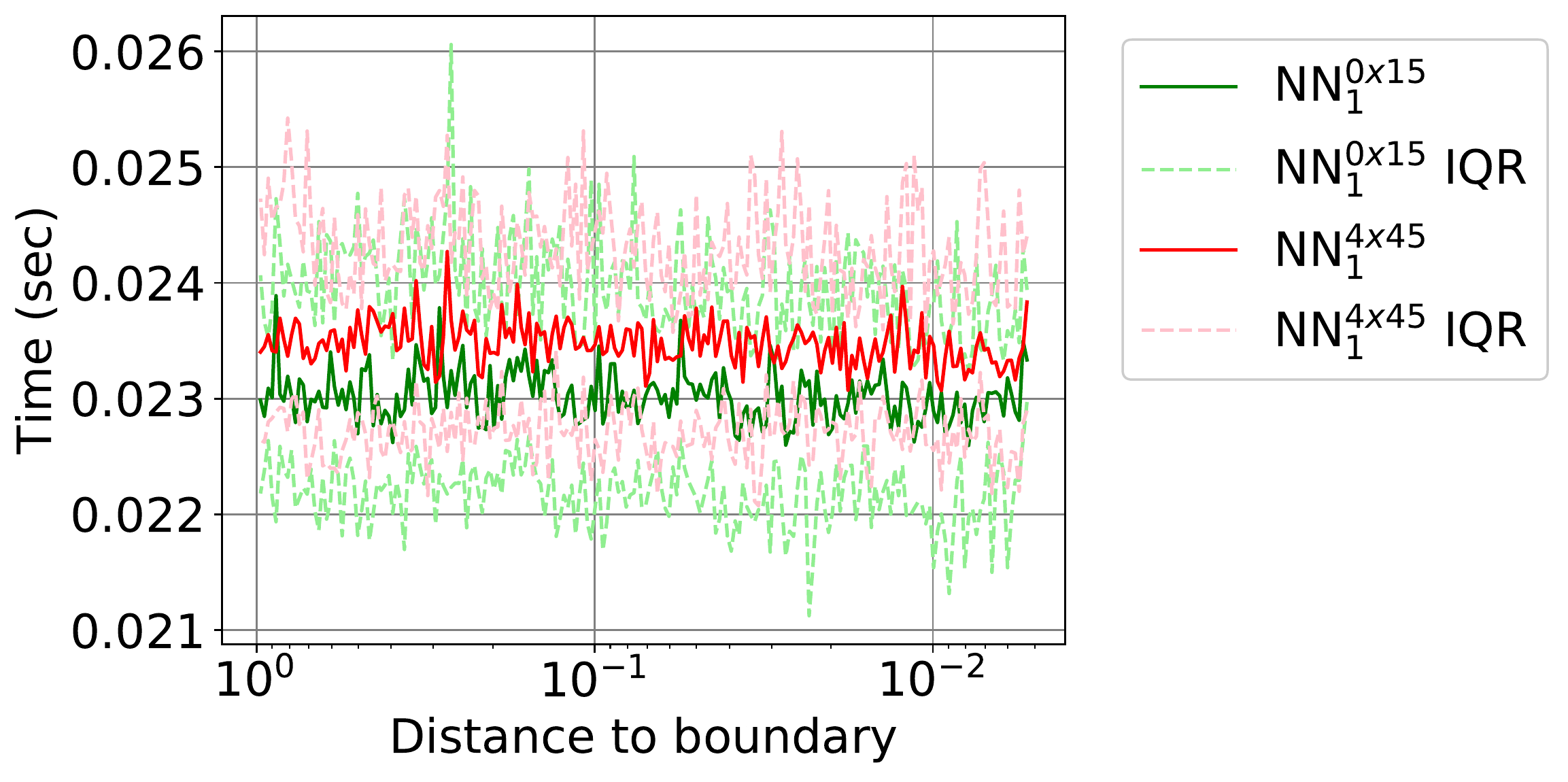} \label{fig:M1_boundaryline_NNtime}}
	\caption{Iteration counts and computation time comparison between two variants of the optimization approach and the network approximations for multiplier calculation with $N=1$. Here the calculations are performed on moments at various distance to the realizable boundary. Neural networks are evaluated point-wisely, which results in longer evaluation time. Figure~\ref{fig:M1_boundaryline_NNtime} shows a zoomed-in version of the medium network computation time for the small ($\texttt{NN}^{0\times15}_1$, green) and large ($\texttt{NN}^{4\times45}_1$, red) networks over 20 trials, where the curves in lighter color denote the upper and lower bounds of the interquartile range (IQR) of the measured time.}
	\label{fig:M1_boundaryline_test}
\end{figure}

Table~\ref{table:M1_boundary_line_test} reports the total time needed for the optimization and network approaches to compute multipliers for 100, 200, 500, 1,000, and 10,000 moments. These results are median values of 20 trails, in which the first and third quartiles are both within $8\%$ deviation from the median. Here the networks are evaluated by calling \texttt{model.predict()} once to evaluate all moments, i.e., batch evaluation. This reduces the overhead for calling \texttt{model.predict()} in network evaluation. The results in Table~\ref{table:M1_boundary_line_test} show that the batch network evaluation is about 6--60x faster than the optimization approach, and the advantage grows as the number of moments increases. This confirms the conjecture that the overhead in calling the Keras function \texttt{model.predict()} is indeed the dominant cost in network evaluation, especially when the number of moments to be evaluated is small.
We note that the results reported in Table~\ref{table:M1_boundary_line_test} are closer to realistic than the ones in Figure~\ref{fig:M1_boundaryline_test}, since batch evaluation is often allowed when solving moment systems with closures, e.g., at a given time step, one can collect moments from all spatial points and perform a batch evaluation to obtain all multipliers needed to evolve the time step.

\begin{table}[h]
	\centering
	\begin{tabular}{|r|c|c|c|c|}
		\toprule
		$\#$ of moments &  $\optQ$ & $\optA$ & $\texttt{NN}^{0\times15}_1$ & $\texttt{NN}^{4\times45}_1$ \\
		\midrule
	100    &      1.76e-1 &  1.54e-1 &  2.46e-2 &  2.51e-2 \\
	200    &      3.56e-1 &  3.08e-1 &  2.74e-2 &  2.93e-2 \\
	500    &      9.13e-1 &  7.98e-1 &  3.21e-2 &  3.94e-2 \\
	1000   &      1.89e0\,\, &  1.71e0\,\, &  4.26e-2 &  5.19e-2 \\
	10000  &      1.99e1\,\, &  1.87e1\,\, &  2.24e-1 &  3.11e-1 \\
	\bottomrule
	\end{tabular}
	\caption{Comparison between the computation time for the optimization approach with quadrature $(\optQ)$ and analytic $(\optA)$ integrals and the batch evaluation time for the small ($\texttt{NN}^{0\times15}_1$) and large ($\texttt{NN}^{4\times45}_1$) networks in the $N=1$ case. The time (sec) for evaluating 100, 200, 500, 1000, and 10000 moments are reported. }
	\label{table:M1_boundary_line_test}
\end{table}

Next, we repeat the comparison between the optimization and neural network approaches in the case $N=2$.
Here we test both approaches with moments moving towards the boundary of the normalized realizable set $\tilde{\cR}$ in four different directions, as shown in Figure~\ref{fig:M2realizable}.
The tested moments are 200 evenly distributed points between the origin and the boundary point in each direction.
The iteration count for $\optQ$ is reported in Figure~\ref{fig:M2realizable_iter} along each direction.
Although analytic integration formula exists for the angular integrations in \eqref{eq:MB_optimization} when $N=2$, we do not include the $\optA$ into the comparison here since there is no substantial difference between $\optQ$ and $\optA$ when $N=1$.
The $\optQ$ optimization time and the network single-point evaluation time are compared in Figure~\ref{fig:M2realizable_time}, where we observe higher optimization time and nearly constant network evaluation time for moments near the realizable boundary. However, as in the $N=1$ case, the single evaluation time for the networks is still higher than the optimization time.
Table~\ref{table:M2_boundary_line_test} reports the total time for computing the multiplier $\hat{\bsalpha}$ from various numbers of given moments, averaged over the four test directions.
As in the results in Table~\ref{table:M1_boundary_line_test}, the network evaluations here are performed in batches, i.e., \texttt{model.predict()} is only called once to evaluate all moments. 
These results are median values of 20 trails, in which the first and third quartiles are both within $9\%$ deviation of the median. 
We observe that, for both the small ($\texttt{NN}^{0\times15}_2$) and large ($\texttt{NN}^{4\times45}_2$) networks, the batch evaluation time is 7--62x faster than the optimization, which is slightly higher than the speedup observed in Table~\ref{table:M1_boundary_line_test} for the $N=1$ case.
The results again confirm that the overhead of calling \texttt{model.predict()} dominants the evaluation cost when evaluating smaller number of moments.
We also expect the speedup of using the neural network closures to become more significant as the moment order $N$ goes higher, in which case the optimization time increases drastically while the network evaluation time is expected to remain in the same order of magnitude.

\begin{figure}[h]
	\centering
	\subfigure[Realizable set $\tilde{\cR}_{\bfm}$]{\includegraphics [width = 0.28\linewidth]{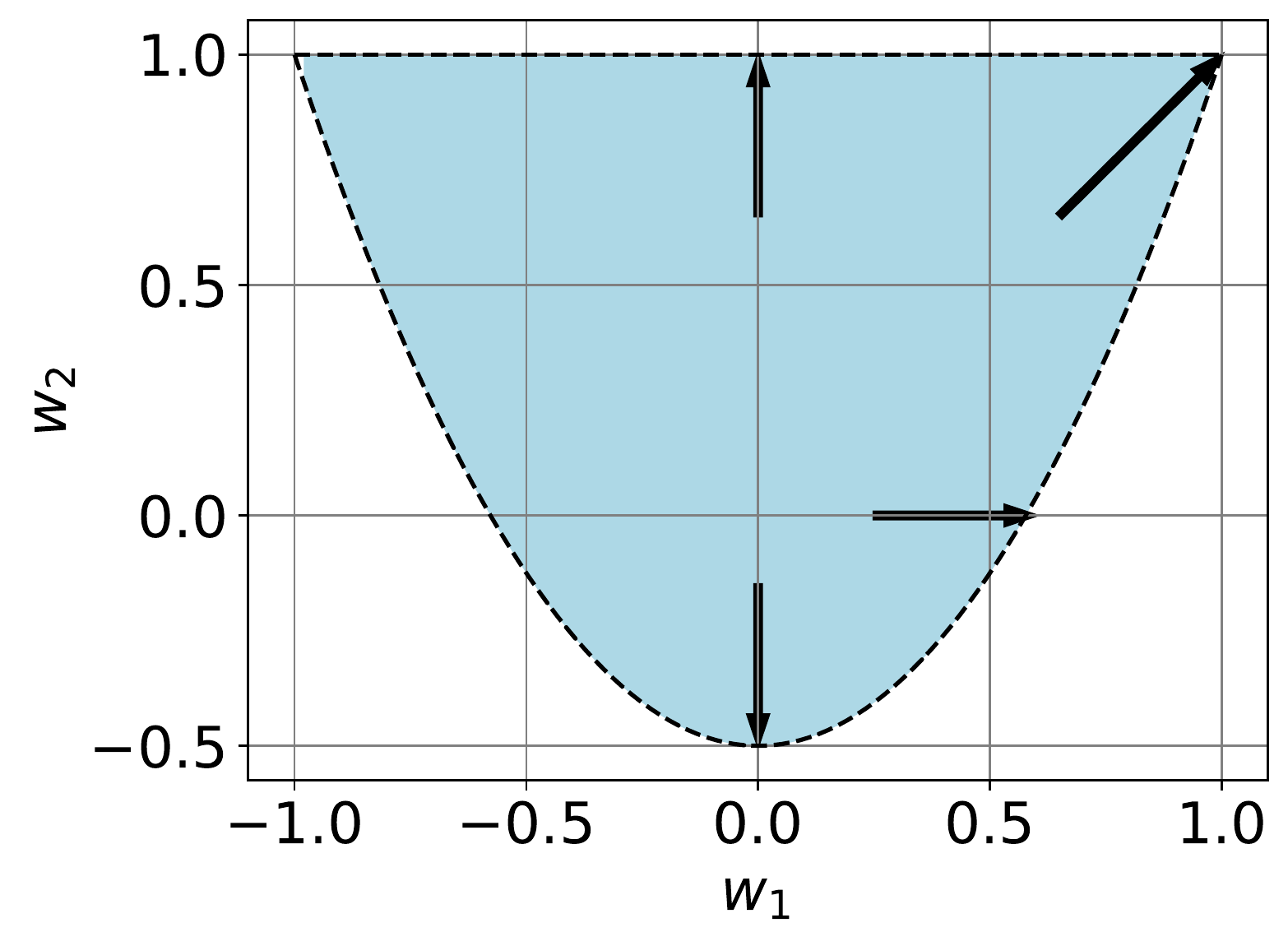} \label{fig:M2realizable}}
	\subfigure[Iteration count]{\includegraphics[width=0.265\linewidth]{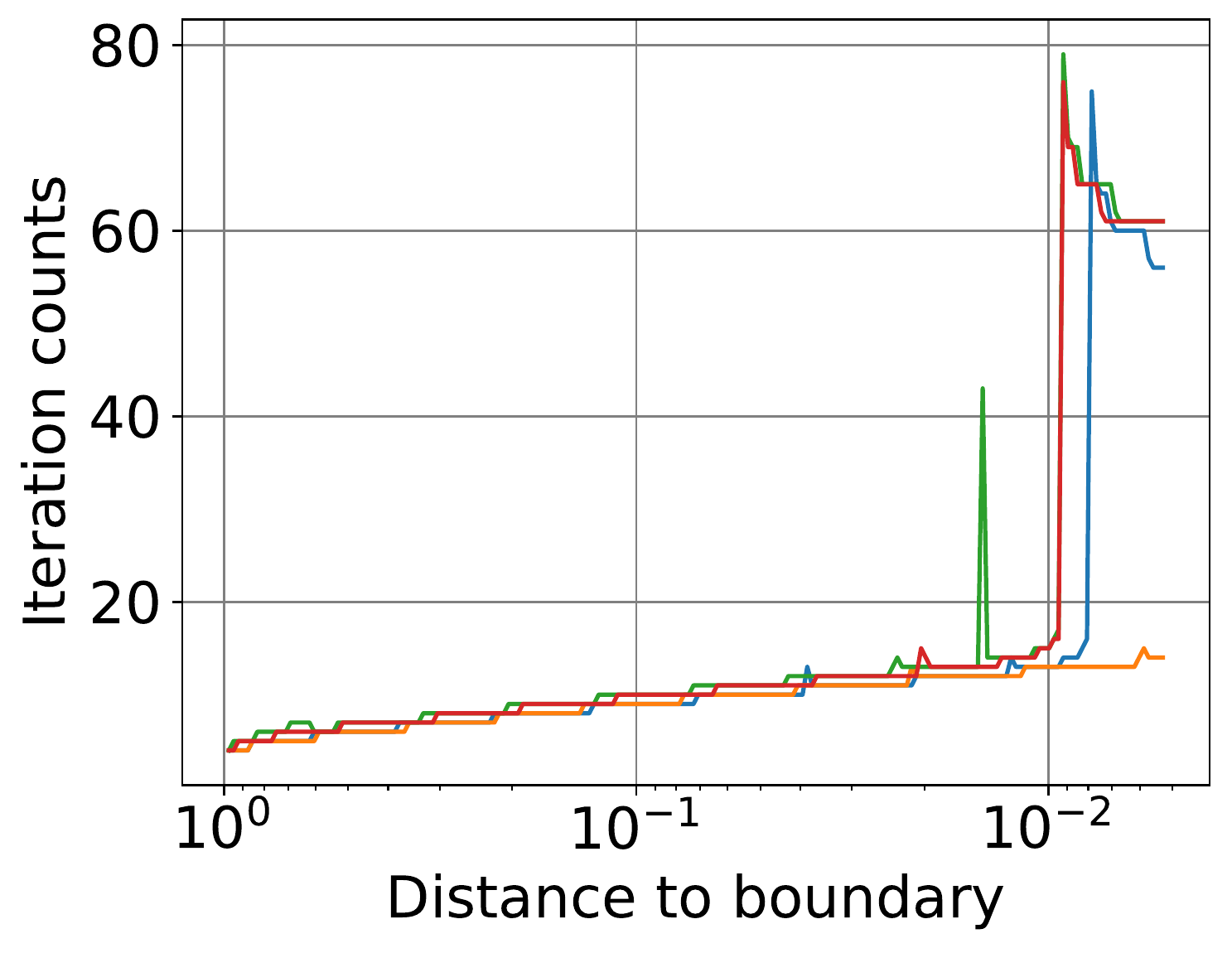} \label{fig:M2realizable_iter}}
	\subfigure[Computation time]{\includegraphics[width=.41\linewidth]{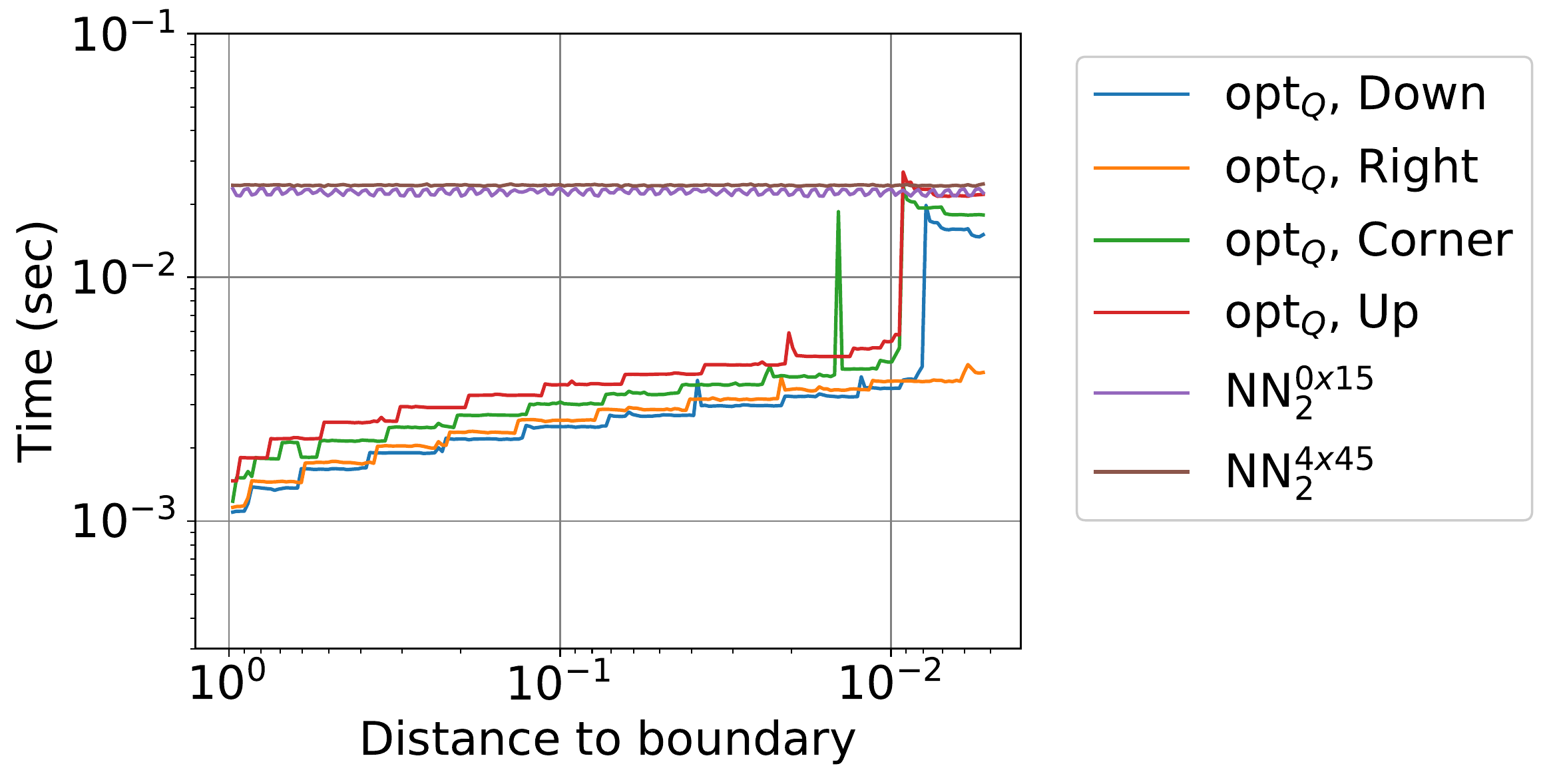} \label{fig:M2realizable_time} }
	\caption{Iteration counts and computation time comparison between the optimization approach ($\optQ$) and two neural network approximations ($\texttt{NN}^{0\times15}_2$ and $\texttt{NN}^{4\times45}_2$)  for multiplier calculation with $N=2$. The calculations are performed along four directions towards the realizable boundary as shown in Figure~\ref{fig:M2realizable}. 
		Neural networks are evaluated point-wisely, which results in longer evaluation time.}
	\label{fig:M2_boundaryline_test}
\end{figure}

\begin{table}[h]
	\centering 
 \begin{tabular}{|r|c|c|c|}
	\toprule
	 $\#$ of moments &    $\optQ$ & $\texttt{NN}^{0\times15}_2$ & $\texttt{NN}^{4\times45}_2$ \\
	\midrule
       100 & 2.18e-1 &      2.84e-2 &      3.08e-2 \\
       200 & 4.25e-1 &      3.01e-2 &      3.46e-2 \\
       500 & 1.07e0\,\, &      3.58e-2 &      4.50e-2 \\
      1000 & 2.15e0\,\, &      4.68e-2 &      6.11e-2 \\
     10000 & 2.20e1\,\, &      2.32e-1 &      3.55e-1 \\
	\bottomrule
\end{tabular}
	\caption{Comparison between the computation time for the optimization approach with quadrature $(\optQ)$ and analytic $(\optA)$ integrals and the batch evaluation time for the small ($\texttt{NN}^{0\times15}_2$) and large ($\texttt{NN}^{4\times45}_2$) networks in the $N=2$ case. The time (sec) for evaluating 100, 200, 500, 1000, and 10000 moments are reported. }
	\label{table:M2_boundary_line_test}
\end{table}

\subsection{Plane source problem}
\label{subsec:planesource}

In this section, the data-driven closures are applied to solve a benchmark problem: the plane source problem.
This problem models the propagation of neutral particles in a purely scattering medium from initial time $t_0=0$ to final time $t_{\textup{final}}$. 
The kinetic model is given in \eqref{eq:kinetic_eqn_slab}, and we consider $t_{\textup{final}}=1$ with a constant scattering crosssection $\sigma_{\text{s}} = 1$.
The initial particle distribution is given by
\begin{equation}\label{eq:initial_condition}
f(t_0=0,x,\mu) = 0.5\delta(x)+f_{\textup{floor}},
\end{equation}
where $\delta$ is the Dirac-delta function and $f_{\textup{floor}}=10^{-8}$ is the floor value for the distribution for safeguarding.
To minimize the effect of the boundaries, the computation domain in $(x,\mu)$ is chosen to be $[x_L,x_R]\times[-1,1]$, where $x_L = -(t_{\textup{final}} +0.1)$ and $x_R = t_{\textup{final}} +0.1$ with boundary conditions
\begin{equation}
f(t,x_L,\mu) = f_{\textup{floor}}\quand f(t,x_R,\mu) = f_{\textup{floor}}.
\end{equation}
Here we solve the corresponding moment equation \eqref{eq:moment_eqn_slab} with moment order $N=1$ and $2$ using Maxwell-Boltzmann entropy-based closure (M$_N$), the P$_N$ closure \cite{Pomraning-1973,Case-Zweifel-1967,Lewis-Miller-1984}, and two data-driven closures --- the spline and neural network closures.
The plane source problem is known to be difficult for the closures since the discontinuous initial distribution leads to moments close to realizable boundary during the evolution. 
We use it to demonstrate that the data-driven closures can serve as efficient approximations to the expensive, optimization-based M$_N$ closure.

The numerical scheme used to solve the moment system is the realizability-preserving kinetic scheme developed in \cite{alldredge2012high} based on earlier work in \cite{Deshpande-1986,Harten-Lax-VanLeer-1983,Perthame-1990,Perthame-1992}.
This scheme discretizes \eqref{eq:moment_eqn_slab} using a second-order finite volume method in space and second-order strong-stability-preserving Runge-Kutta (SSP-RK2) method in time.
In the following tests, we solve \eqref{eq:moment_eqn_slab} on a uniform spatial grid of $100$ cells, and the boundary condition is implemented using two ghost cells on each side of the boundary.
The time step is given by $\dt = 0.95(\frac{2}{2+\theta}) \dx$, where $\theta$ is a parameter in the minmod limiter used in the numerical scheme. We use $\theta=2$ throughout the tests.

\paragraph{Closure comparison for $N=1$.}
Figure~\ref{fig:planesourceN1} shows the solutions to the plane source problem at time  given by the tested closures with moment order $N=1$.
The M$_1$ closure is implemented using both the $\optQ$ and $\optA$ approaches discussed in Section~\ref{subsec:realizable_boundary_test}, where the a 10-point Gauss-Legendre quadrature rule is used in $\optQ$.
We show the solutions from these two approaches in the same plot as they are indistinguishable. 
The solutions of spline closures are shown in Figure~\ref{fig:splineN1}, where we use the coarser ($\texttt{S}^{30}$) and finer ($\texttt{S}^{130}$) spline approximations reported in Table~\ref{table:constructed_summary_by_error_M1}.
The results of the network closures are plotted in Figure~\ref{fig:networkN1}, where the smaller ($\texttt{NN}^{1\times15}_{1}$) and larger ($\texttt{NN}^{5\times30}_{1}$) networks in Table~\ref{table:constructed_summary_by_error_M1} are tested. To preserve the symmetry of the solution, the symmetric network approximation given in \eqref{eq:symmetry} is used here.
From these results, we confirm that, with sufficiently accurate approximations, the data-driven closures can provide quality solutions similar to the ones from the more expensive M$_1$ closures.

\begin{figure}[h]
	\centering 
	\subfigure[M$_1$, $\optQ$/$\optA$]{\includegraphics[width=.37\linewidth]{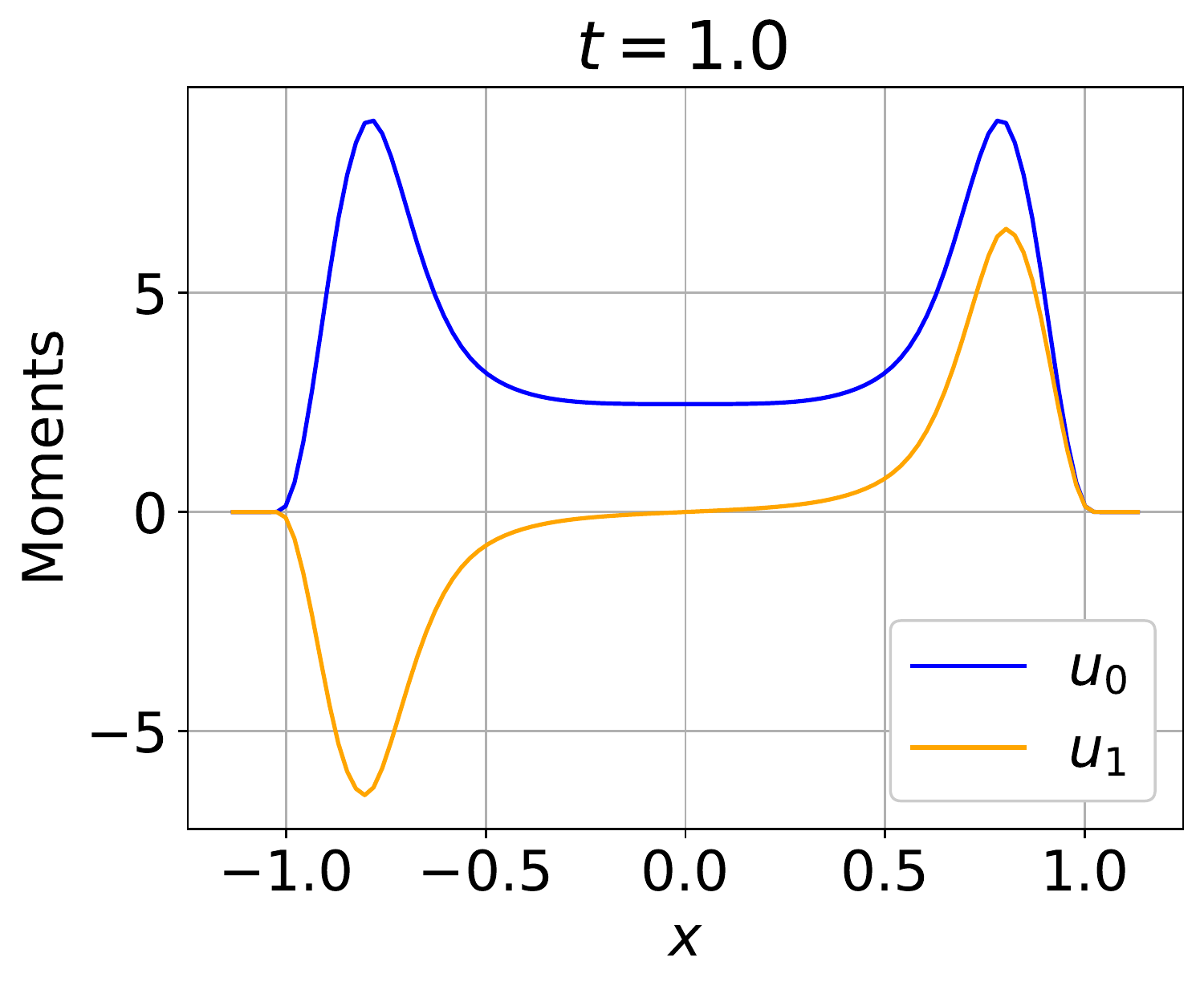}}~~~
	\addtocounter{subfigure}{1}
	\subfigure[Spline closures]{\includegraphics[width= 0.52\linewidth]{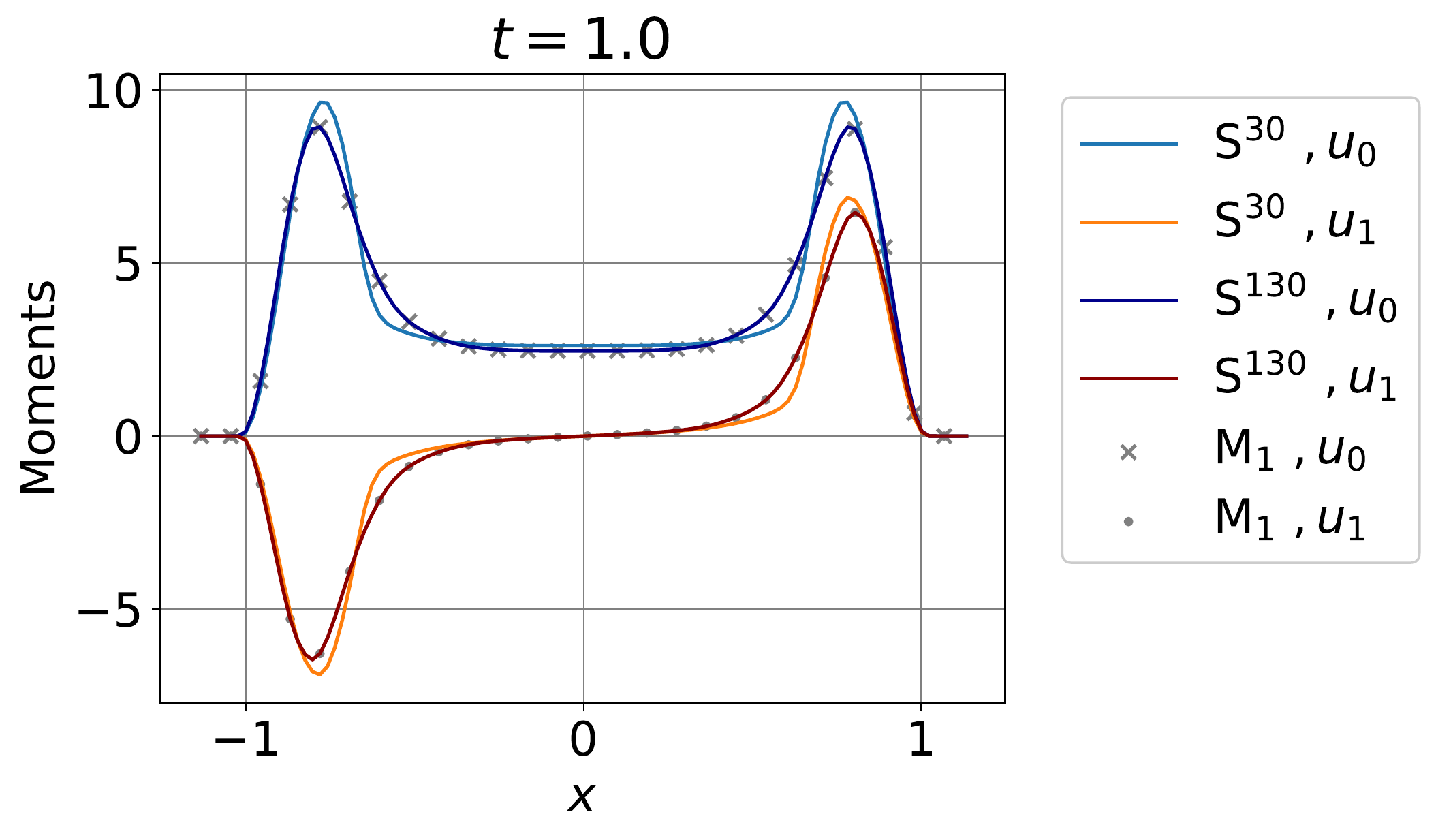} \label{fig:splineN1} }	~~~~~~\\
	\addtocounter{subfigure}{-2}
	\subfigure[P$_1$]{\includegraphics[width=.38\linewidth]{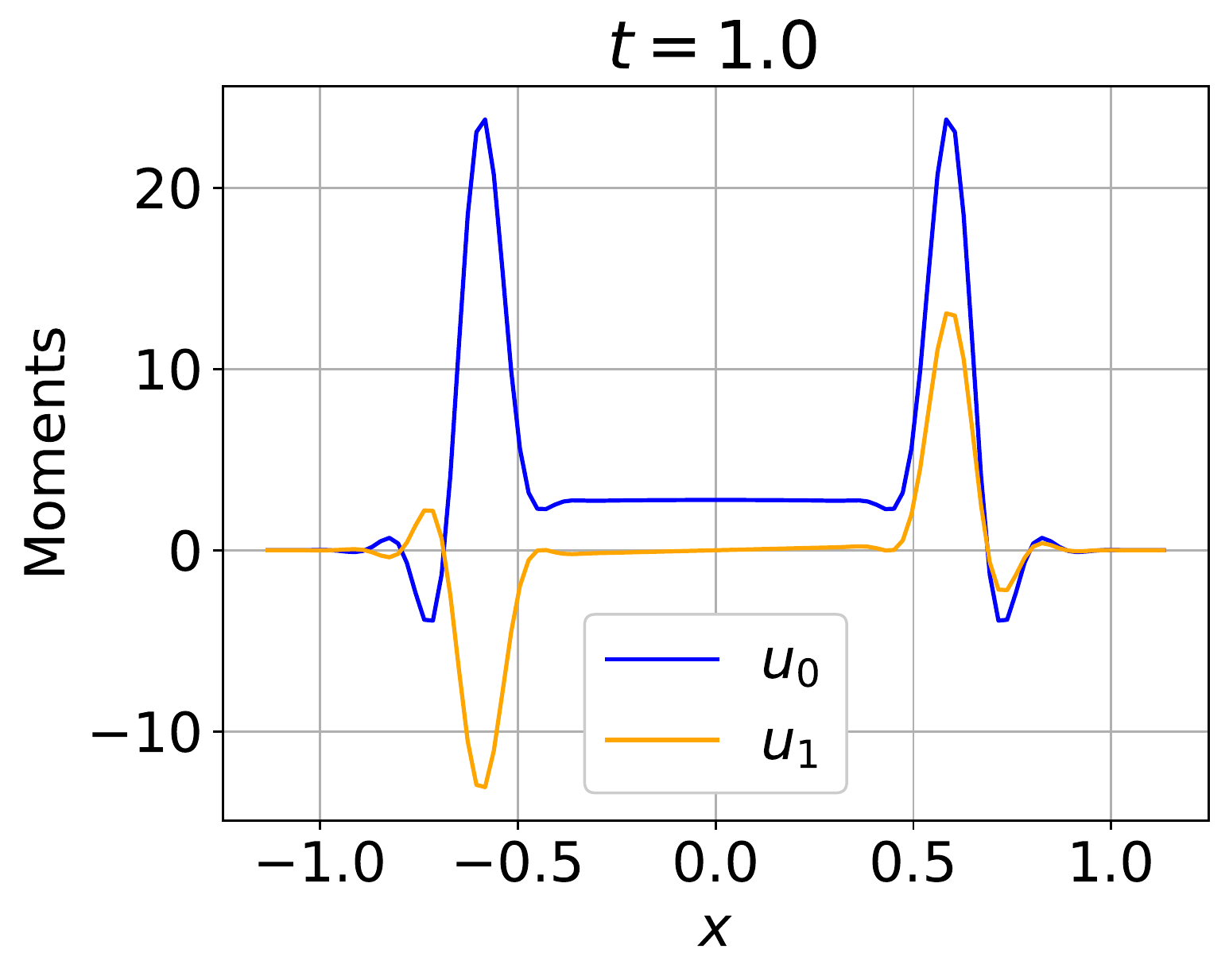}}	~~~
	\addtocounter{subfigure}{1}
	\subfigure[Symmetric networks closures]{\includegraphics[width= 0.56\linewidth]{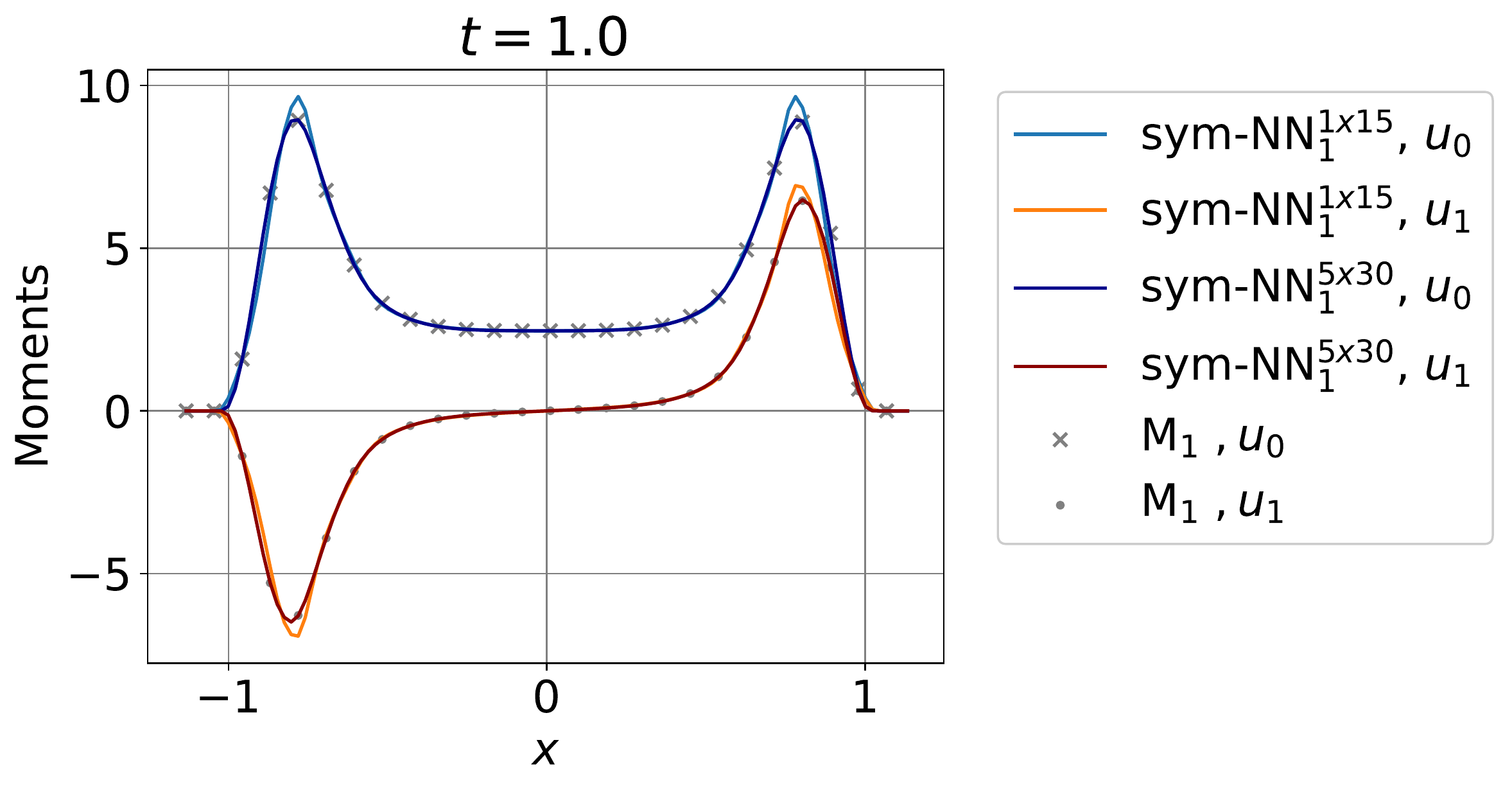} \label{fig:networkN1} }
	\caption{Solutions to the plane source problem at $t_{\text{final}} = 1$ using various moment closures with $N=1$.  Figures~\ref{fig:splineN1} and \ref{fig:networkN1} show that the solutions from the larger spline ($\texttt{S}^{130}$) and network ($\texttt{sym-NN}^{5\times30}_1$) are nearly identical to the reference M$_1$ solution. }
	\label{fig:planesourceN1}
\end{figure}

Table~\ref{table:M1_plane_test_table} reports the computation time for the tested closures and the relative error of the solutions from data-driven closures at the final time $t_{\textup{final}}=1$.
Here we define the error as
\begin{equation}\label{eq:sol_error}
\Err_{\bu}:=\frac{\|\bu_{\textup{a}}(t_{\textup{final}},\cdot) - \bu_{\text{ref}}\|_{L^2([x_L,x_R])}}{\|\bu_{\text{ref}}\|_{L^2([x_L,x_R])}},\quad\text{with norm }
\|\bw\|_{L^2([x_L,x_R])}:=\big(\int_{x_L}^{x_R} \|\bw\|_2^2 \, dx\big)^{\frac{1}{2}}\:,
\end{equation}
where the reference solution $\bu_{\text{ref}}$ is chosen to be the M$_N$ ($\optQ$) solution at final time, since the data-driven closures aim to approximate the M$_N$ closure.
The computation time for the P$_1$ closure serves as a baseline in the following comparison, since the cost of computing the P$_1$ closure is minimal.
We observe from Table~\ref{table:M1_plane_test_table} that, for both the spline and network closures, increasing the number of parameters improves the solution accuracy with minimal effect in the computation time. When using the finer spline ($\texttt{S}^{130}$) and larger network ($\texttt{NN}^{5\times30}_{1}$) closures, the errors are of order $10^{-3}$.
The fast evaluation of the spline closures results in a 183x speedup comparing to the M$_1$ closure. Despite the slower evaluation of the symmetric network closures due to the larger number of parameters and the overhead for calling the \texttt{model.predict()} Keras function, the symmetric network closures is still 2.6x faster than the M$_1$ closure.

\begin{table}[h]
\centering 
\begin{tabular}{|c|c|c|c|c|c|c|c|}
	\toprule
	Closure &    M$_1$ $\optQ$ & M$_1$ $\optA$ &   P$_1$ & $\texttt{sym-NN}^{1\times15}_1$ &  $\texttt{sym-NN}^{5\times30}_1$ & $\texttt{S}^{30}$ & $\texttt{S}^{130}$ \\
	\midrule
	Time (sec) & 1.66e1 & 1.62e1 & 2.15e-2 &            6.09e0\,\, &            6.24e0\,\, &  8.71e-2 &   8.86e-2 \\
	$\Err_{\bu}$      & --- & --- & --- &            6.33e-2 &            2.85e-3 &  1.07e-1 &   1.49e-3 \\
	\bottomrule
\end{tabular}
\caption{Computation time (sec) and relative errors in the plane source test using various closures with $N=1$. The error $\Err_{\bu}$ (see \eqref{eq:sol_error} for definition) measures the relative $L^2$ error in the final time solutions from data-driven closures with respect to the one from the M$_1$ $\optQ$ closure. } 
\label{table:M1_plane_test_table}
\end{table}

\paragraph{Closure comparison for $N=2$.}
We then test the closures on the plane source problem for moment order $N=2$. The solutions for the tested closures are plotted in Figure~\ref{fig:planesourceN2} and the computation time and relative errors are reported in Table~\ref{table:M2_error}.
Here we only compare the M$_2$ ($\optQ$), P$_2$, and the symmetric network closures, since the spline approximation in Section~\ref{subsubsec:splines} is limited to one dimension. 
Figure~\ref{fig:planesourceN2} confirms that the solutions from symmetric network closures successfully capture the qualitative behavior of the M$_2$ solution. 
Using the larger network ($\texttt{sym-NN}^{4\times45}_2$), the network closure solution is nearly identical to the reference M$_2$ solution.

\begin{figure}[h]
	\centering
	\subfigure[M$_2$, $\optQ$]{\includegraphics[width=.458\linewidth]{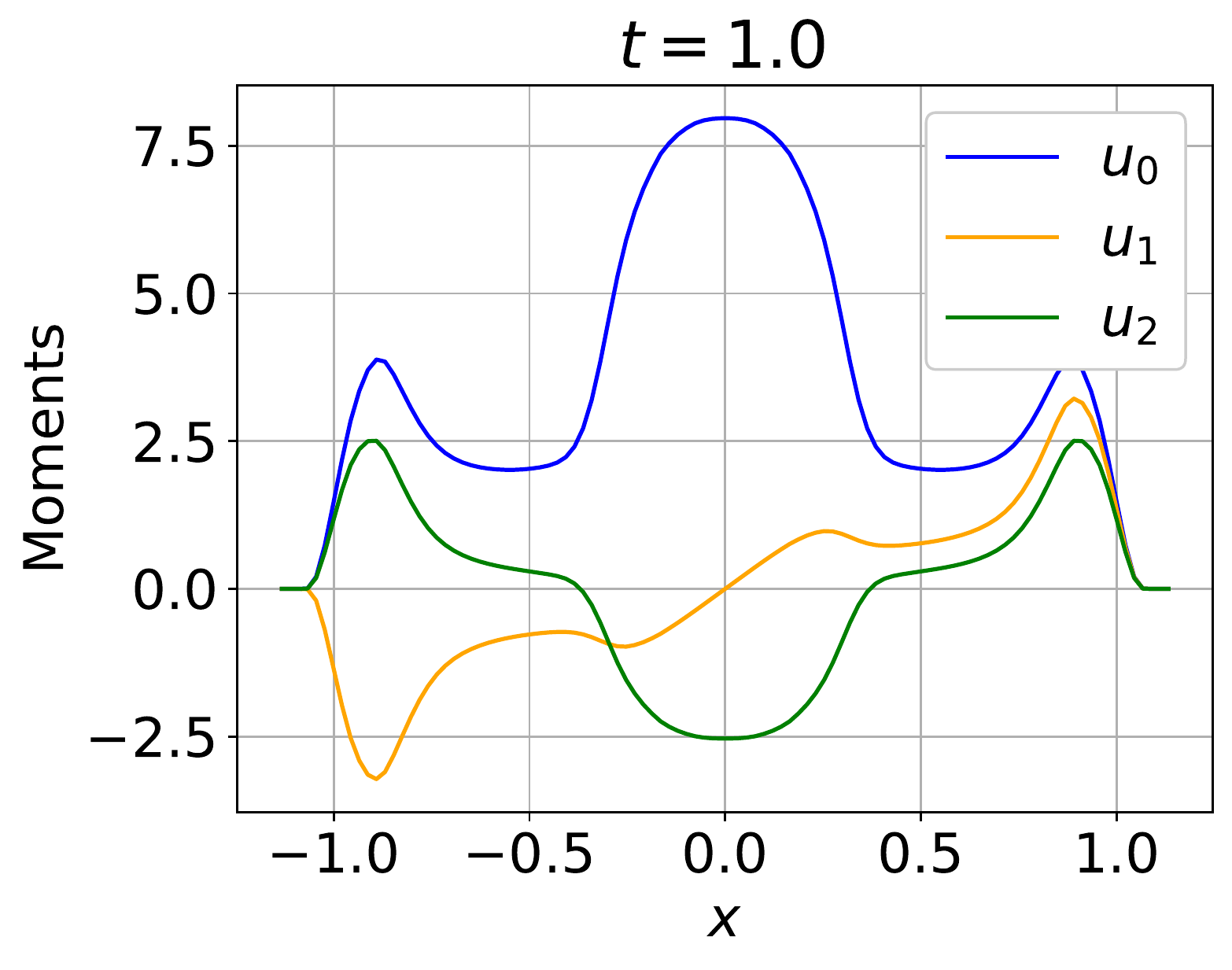}}~~~~~~~~
	\subfigure[P$_2$]{\includegraphics[width=.45\linewidth]{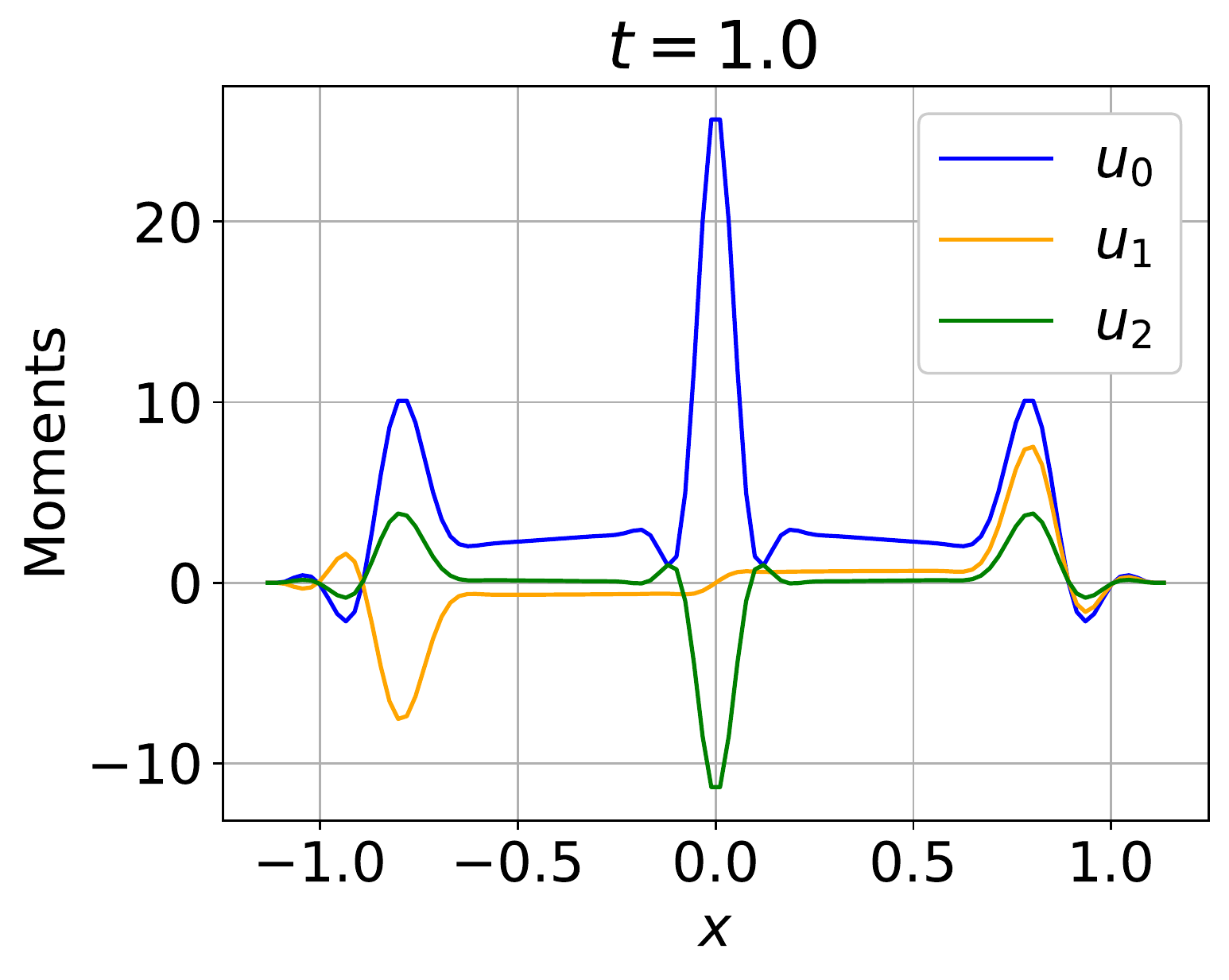}}
	\subfigure[Symmetric network closures]{	\centering \includegraphics[width=0.8\linewidth]{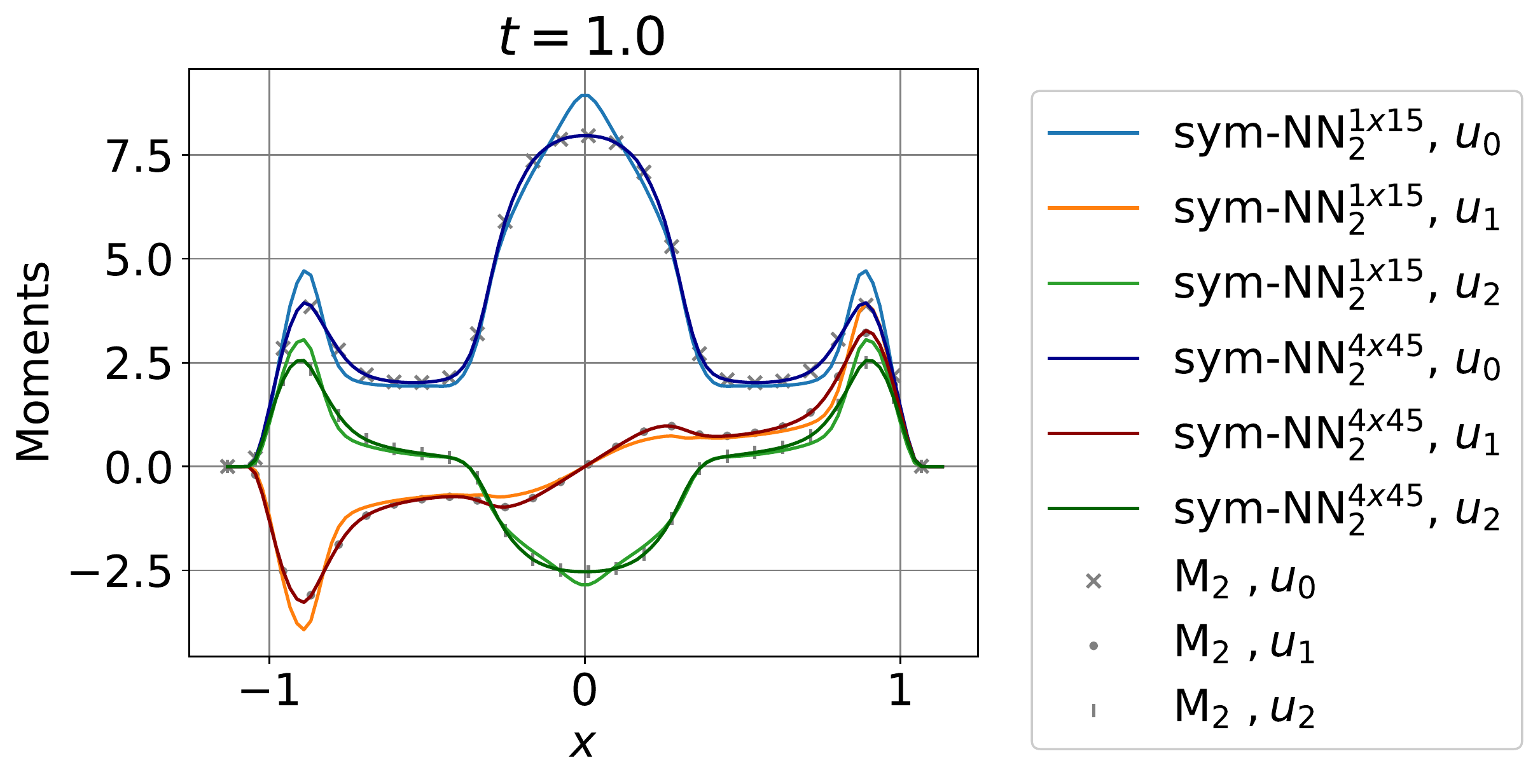} \label{fig:M2_plane_net}}
	\caption{Solutions to the plane source problem at $t_{\text{final}} = 1$ using various moment closures with $N=2$.  Figure~\ref{fig:M2_plane_net} shows that the solution from the larger network ($\texttt{sym-NN}^{4\times45}_2$) is nearly identical to the reference M$_2$ solution. }
	\label{fig:planesourceN2}	
\end{figure}

The computation time in Table~\ref{table:M2_error} shows a 3.2x speedup when replacing the M$_2$ closure by the network closure, which is slightly larger than the 2.6x speedup observed in Table~\ref{table:M1_plane_test_table} for the case $N=1$.
We expect the speedup for using the network closure to increase as the moment order $N$ goes higher, because the optimization problem \eqref{eq:MB_optimization} becomes much harder to solve for large $N$ while the network evaluation time does not appear to have a strong dependency on the value of $N$.

\begin{table}[h]
	\centering
	\begin{tabular}{|c|c|c|c|c|}
		\toprule
		Closure &    M$_2$ $\optQ$ &       P$_2$ &  $\texttt{sym-NN}^{1\times15}_2$ &  $\texttt{sym-NN}^{4\times45}_2$\\
		\midrule
		Time (sec) & 1.90e1 & 2.06e-2 &        5.78e0\,\, &        6.01e0\,\,  \\
		$\Err_{\bu}$      & --- & --- &        9.66e-2 &        5.49e-3  \\
		\bottomrule
	\end{tabular}
\caption{Computation time (sec) and relative errors in the plane source test using various closures with $N=2$. The error $\Err_{\bu}$ (see \eqref{eq:sol_error} for definition) measures the relative $L^2$ error in the final time solutions from data-driven closures with respect to the one from the M$_2$ $\optQ$ closure.} 
	\label{table:M2_error}
\end{table}

\paragraph{Closure comparison on a refined spatial grid for $N=2$.}
The slight differences between the computation time for the small and large networks in Tables~\ref{table:M1_plane_test_table} and \ref{table:M2_error} suggest that the dominant cost in network evaluations is the overhead for calling Keras function \texttt{model.predict()}, as discussed in Section~\ref{subsec:realizable_boundary_test}. To minimize the effect of the overhead, we repeat the comparison test of the closures but on a refined spatial grid with the number of spatial cells increased from 100 to 1000. In this test, each call of \texttt{model.predict()} evaluates 1000, rather than 100, samples of moments, thus it could give a more accurate comparison between the optimization cost and the actual network evaluation cost. To avoid numerical difficulties, we also replace the Dirac-delta function initial condition in \eqref{eq:initial_condition} with a smoother isotropic initial condition
\begin{equation}\label{eq:initial_condition_smooth}
f(t_0=0,x,\mu) = 0.5\chi_{[-\frac{1}{2},\frac{1}{2}]}(x) \cos^2(\pi  x) + f_{\textup{floor}},
\end{equation}
where $\chi_{[a,b]}(x) = 1$ if $x\in[a,b]$ and $\chi_{[a,b]}(x) = 0$ otherwise.
Table~\ref{table:M2_error_smooth} shows the relative error and computation time of the compared closures using the refined spatial grid with 1000 cells and the smooth initial condition \eqref{eq:initial_condition_smooth}. In Table~\ref{table:M2_error_smooth}, a 9.7x speedup for using the large network $\texttt{sym-NN}^{4\times45}_2$ over the M$_2$ closure is reported, which is higher than the 3.2x speedup observed in Table~\ref{table:M2_error}. This result suggests that the advantage in computation time of the network closures can be further improved via an efficient implementation of the network evaluation function with minimal overhead.

\begin{table}[h]
	\centering 
\begin{tabular}{|c|c|c|c|c|}
	\toprule
	 Closure & M$_2$ $\optQ$ &       P$_2$ &  $\texttt{sym-NN}^{1\times15}_2$ &  $\texttt{sym-NN}^{4\times45}_2$ \\
	\midrule
	Time (sec)  & 1.06e3 & 2.96e-1 &        9.03e1\,\, &        1.09e2\,\,\\
	$\Err_{\bu}$     & --- & --- &        6.43e-2 &        1.92e-3 \\
	\bottomrule
\end{tabular}
\caption{Computation time (sec) and relative errors in the test with 1000 spatial cells and smooth initial condition \eqref{eq:initial_condition_smooth} using various closures with $N=2$. The error $\Err_{\bu}$ (see \eqref{eq:sol_error} for definition) measures the relative $L^2$ error in solutions at final time $t_{\textup{final}}=1$ from data-driven closures with respect to the one from the M$_2$ $\optQ$ closure.} 
	\label{table:M2_error_smooth}
\end{table}

\paragraph{Correlation between network moment error $\Err_{\bw}^{\test}$ and solution error $\Err_{\bu}$.}
In Figure~\ref{fig:error_correlation}, we explore the correlation between the test moment error $\Err_{\bw}^{\test}$ of the network approximations and the relative $L^2$ error $\Err_{\bu}$ in the plane source solution from the network closures when $N=2$. The data shown in Figure~\ref{fig:error_correlation} are collected from the small ($\texttt{sym-NN}^{1\times15}_2$) and large $(\texttt{sym-NN}^{4\times45}_2)$ networks considered in Table~\ref{table:M2_error} as well as several other networks of different sizes. 
These results show that the test moment error $\Err_{\bw}^{\test}$ is indeed a good proxy of the error $\Err_{\bu}$ in the final solution, which justifies the use of the moment reconstruction error in the loss function \eqref{eq:loss_function} when training the neural network approximations.
\begin{figure}[h]
	\centering 
    \includegraphics[width = 0.60\linewidth]{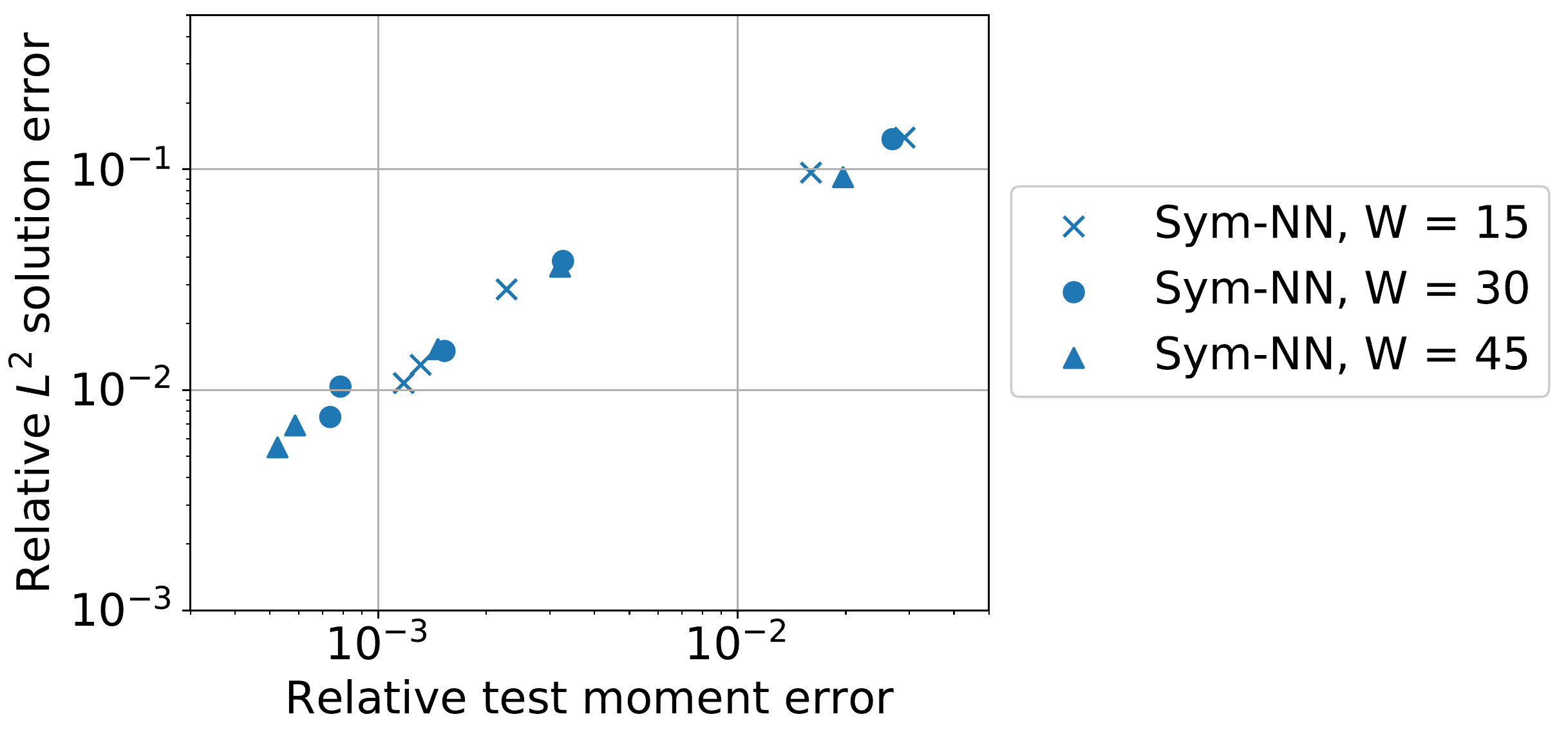} 
	\caption{Test moment errors $\Err_{\bw}^{\test}$ versus $L^2$ error of the plane source solution $\Err_{\bu}$ from the various symmetric neural network closures for $N = 2$. Closures of various depths with width $W = $ 15, 30, and 45 are denoted as blue crosses, circles, and triangles.}
\label{fig:error_correlation}
\end{figure}

\section{Summary and discussion}
\label{sec:conclusion}

We have developed data-driven closures that approximate the map between the moments and the entropy of the moment system for kinetic equations. The approximations are constructed to be convex and twice differentiable to ensure entropy dissipation and hyperbolicity in the moment system. We have shown that convexity of the entropy approximation is preserved in the extension from a normalized realizable set to the full realizable set of moments, which allows for the strategy of building a convex entropy approximation on the bounded normalized realizable set and then extending the approximation to the unbounded full realizable set.
We have focused on the Maxwell-Boltzmann entropy and considered two approaches --- rational cubic splines and neural networks --- to approximate the Maxwell-Boltzmann entropy.
The convexity-preserving cubic spline approach gives fast, accurate, and data-efficient approximations that are convex by construction, but this approach is limited to the case that the normalized realizable set is an interval, which generally implies the restriction of moment order $N=1$ in the simplified slab geometry.
On the other hand, the neural network approach does not have the restrictions on the moment order or geometry, but it requires a substantial amount of data in the training process to build an accurate approximation with convexity enforced on the data points.
We have tested the data-driven closures on the plane source benchmark problem in slab geometry and compared them to the standard, optimization-based M$_N$ closures.
Numerical results indicate that the data-driven closures are able to provide solutions that capture the qualitative behavior of the M$_N$ solutions while requiring much less computation time.

We plan to extend this work in various directions: (i) adopting neural network architectures that guarantee convexity of the network approximation, e.g., \cite{amos2017input} and investigating tools or implementations that allow for faster network evaluation;
(ii) combining the data-driven closures with the entropy regularization scheme developed in \cite{alldredge2019regularized} to address the issue that multipliers goes unbounded as moments approaching realizable boundary; (iii) investigating the extension of proposed data-driven closures to more complex problems that are in higher dimensions with higher moment order and potentially from other applications, such as charge transport or plasma simulations; (iv) exploring other data-driven entropy approximations, such as high-dimensional splines, that can potentially be more data-efficient with faster evaluation; and (v) developing strategies to circumvent the degeneracy issue that the entropy minimization problem has no solution for some realizable moments (see \cite{alldredge2019regularized} and references therein for discussion), which is not addressed in either the regularization scheme in \cite{alldredge2019regularized} or the data-driven closures considered in this paper.


\bibliographystyle{plain}
\bibliography{reference}

\end{document}